\documentclass[11pt,a4paper]{article}
\usepackage{amsmath}
\usepackage{amsthm}
\usepackage{amssymb}
\usepackage{amscd}
\usepackage{graphicx}
\usepackage{epsfig}
\usepackage[matrix,arrow,curve]{xy}

\usepackage{latexsym}
\usepackage{amsfonts}
\input xy
\usepackage{tikz}
\usetikzlibrary{arrows,chains,matrix,positioning,scopes,snakes}
\makeatletter
\tikzset{join/.code=\tikzset{after node path={%
\ifx\tikzchainprevious\pgfutil@empty\else(\tikzchainprevious)%
edge[every join]#1(\tikzchaincurrent)\fi}}}
\makeatother

\tikzset{>=stealth',every on chain/.append style={join},
         every join/.style={->}}
\tikzset{
    >=stealth',
    punkt/.style={
           rectangle,
           rounded corners,
           draw=black, very thick,
           text width=6.5em,
           minimum height=2em,
           text centered},
    pil/.style={
           ->,
           thick,
           shorten <=2pt,
           shorten >=2pt,}
}

\xyoption{all} \tolerance=500
\newtheorem{thm}{Theorem}[section]
\newtheorem{prop}[thm]{Proposition}
\newtheorem{lem}[thm]{Lemma}
\newtheorem{cor}[thm]{Corollary}

\theoremstyle{definition}
\newtheorem{example}[thm]{Example}
\newtheorem{remark}[thm]{Remark}
\newtheorem{definition}[thm]{Definition}

\font\black=cmbx10 \font\sblack=cmbx7 \font\ssblack=cmbx5 \font\blackital=cmmib10  \skewchar\blackital='177
\font\sblackital=cmmib7 \skewchar\sblackital='177 \font\ssblackital=cmmib5 \skewchar\ssblackital='177
\font\sanss=cmss12 \font\ssanss=cmss8 scaled 900 \font\sssanss=cmss8 scaled 600 \font\blackboard=msbm10
\font\sblackboard=msbm7 \font\ssblackboard=msbm5 \font\caligr=eusm10 \font\scaligr=eusm7 \font\sscaligr=eusm5

\font\bsymb=cmsy10 scaled\magstep2
\def\all#1{\setbox0=\hbox{\lower1.5pt\hbox{\bsymb
       \char"38}}\setbox1=\hbox{$_{#1}$} \box0\lower2pt\box1\;}
\def\exi#1{\setbox0=\hbox{\lower1.5pt\hbox{\bsymb \char"39}}
       \setbox1=\hbox{$_{#1}$} \box0\lower2pt\box1\;}

\newfam\bifam
\textfont\bifam=\blackital \scriptfont\bifam=\sblackital \scriptscriptfont\bifam=\ssblackital

\newfam\blfam
\textfont\blfam=\black \scriptfont\blfam=\sblack \scriptscriptfont\blfam=\ssblack

\newfam\bbfam
\textfont\bbfam=\blackboard \scriptfont\bbfam=\sblackboard \scriptscriptfont\bbfam=\ssblackboard

\newfam\ssfam
\textfont\ssfam=\sanss \scriptfont\ssfam=\ssanss \scriptscriptfont\ssfam=\sssanss
\def\sss#1{{\fam\ssfam\relax#1}}

\newfam\clfam
\textfont\clfam=\caligr \scriptfont\clfam=\scaligr \scriptscriptfont\clfam=\sscaligr

\def\pmb#1{\setbox0\hbox{${#1}$} \copy0 \kern-\wd0 \kern.2pt \box0}
\def\pmbb#1{\setbox0\hbox{${#1}$} \copy0 \kern-\wd0
      \kern.2pt \copy0 \kern-\wd0 \kern.2pt \box0}
\def\pmbbb#1{\setbox0\hbox{${#1}$} \copy0 \kern-\wd0
      \kern.2pt \copy0 \kern-\wd0 \kern.2pt
    \copy0 \kern-\wd0 \kern.2pt \box0}
\def\pmxb#1{\setbox0\hbox{${#1}$} \copy0 \kern-\wd0
      \kern.2pt \copy0 \kern-\wd0 \kern.2pt
      \copy0 \kern-\wd0 \kern.2pt \copy0 \kern-\wd0 \kern.2pt \box0}
\def\pmxbb#1{\setbox0\hbox{${#1}$} \copy0 \kern-\wd0 \kern.2pt
      \copy0 \kern-\wd0 \kern.2pt
      \copy0 \kern-\wd0 \kern.2pt \copy0 \kern-\wd0 \kern.2pt
      \copy0 \kern-\wd0 \kern.2pt \box0}

\mathchardef\za="710B  %\alpha
\mathchardef\zb="710C  %\beta
\mathchardef\zg="710D  %\gamma
\mathchardef\zd="710E  %\delta
\mathchardef\zve="710F %\epsilon
\mathchardef\zz="7110  %\zeta
\mathchardef\zh="7111  %\eta
\mathchardef\zvy="7112 %\theta
\mathchardef\zi="7113  %\iota
\mathchardef\zk="7114  %\kappa
\mathchardef\zl="7115  %\lambda
\mathchardef\zm="7116  %\mu
\mathchardef\zn="7117  %\nu
\mathchardef\zx="7118  %\xi
\mathchardef\zp="7119  %\pi
\mathchardef\zr="711A  %\rho
\mathchardef\zs="711B  %\sigma
\mathchardef\zt="711C  %\tau
\mathchardef\zu="711D  %\upsilon
\mathchardef\zvf="711E %\phi
\mathchardef\zq="711F  %\chi
\mathchardef\zc="7120  %\psi
\mathchardef\zw="7121  %\omega
\mathchardef\ze="7122  %\varepsilon
\mathchardef\zy="7123  %\vartheta
\mathchardef\zf="7124  %\varomega
\mathchardef\zvr="7125 %\varrho
\mathchardef\zvs="7126 %\varsigma
\mathchardef\zf="7127  %\varphi
\mathchardef\zG="7000  %\Gamma
\mathchardef\zD="7001  %\Delta
\mathchardef\zY="7002  %\Theta
\mathchardef\zL="7003  %\Lambda
\mathchardef\zX="7004  %\Xi
\mathchardef\zP="7005  %\Pi
\mathchardef\zS="7006  %\Sigma
\mathchardef\zU="7007  %\Upsilon
\mathchardef\zF="7008  %\Phi
\mathchardef\zW="700A  %\Omega

\newcommand{\be}{\begin{equation}}
\newcommand{\ee}{\end{equation}}
\newcommand{\raa}{\rightarrow}

\newcommand{\bea}{\begin{eqnarray}}
\newcommand{\eea}{\end{eqnarray}}
\newcommand{\beas}{\begin{eqnarray*}}
\newcommand{\eeas}{\end{eqnarray*}}
\def\*{{\textstyle *}}

\newcommand{\w}{\wedge}
\newcommand{\nn}{\nonumber}

\newcommand{\ti}{\times}

\def\ran{\rangle}

\def\cA{{\cal A}}

\def\cN{{\cal N}}

\def\cJ{{\cal J}}

\def\cO{{\cal O}}
\def\cU{{\cal U}}
\def\cF{\mathcal{F}}

\def\sT{{\sss T}}

\def\cM{{\cal M}}

\newdir{|>}{%
!/4.5pt/@{|}*:(1,-.2)@^{>}*:(1,+.2)@_{>}}

\newcommand{\eps}{\varepsilon}

\newcommand{\mZ}{\mathbb{Z}_2}

\newcommand{\p}{\partial}
\newcommand{\la}{\langle}
\newcommand{\ra}{\rangle}

\newcommand{\Ci}{C^{\infty}}

\newcommand{\N}{\mathbb{N}}
\newcommand{\Z}{\mathbb{Z}}
\newcommand{\R}{\mathbb{R}}

\newcommand{\lp}{\left(}
\newcommand{\rp}{\right)}

\newcommand{\op}[1]{\!\!\mathop{\rm ~#1}\nolimits}

\begin{document}
\title{\bf $\mathbb{Z}_2^n$-Supergeometry I\\ Manifolds and Morphisms}
\date{}
\author{Tiffany Covolo, Janusz Grabowski, and Norbert Poncin}

\maketitle

\begin{abstract}{In Physics and in Mathematics $\Z_2^n$-gradings, $n\ge 2$, do appear quite frequently. The corresponding sign rules are determined by the `scalar product' of the involved $\Z_2^n$-degrees. The present paper is the first of a series on $\Z_2^n$-Supergeometry. The new theory exhibits challenging differences with the classical one: nonzero degree even coordinates are not nilpotent, and even (resp., odd) coordinates do not necessarily commute (resp., anticommute) pairwise (the parity is the parity of the total degree). It is based on the hierarchy: `$\,\Z_2^0$-Supergeometry (classical differential Geometry) contains the germ of $\Z_2^1$-Supergeometry (standard Supergeometry), which in turn contains the sprout of $\Z_2^2$-Supergeometry, etc.' The $\Z_2^n$-supergeometric viewpoint provides deeper insight and simplified solutions; interesting relations with Quantum Field Theory and Quantum Mechanics are expected. In this article, we define $\Z_2^n$-supermanifolds and provide examples in the atlas, the ringed space and coordinate settings. We thus show that formal series are the appropriate substitute for nilpotency. Moreover, the category of $\Z_2^n$-supermanifolds is closed with respect to the tangent and cotangent functors. The fundamental theorem describing supermorphisms in terms of coordinates is extended to the $\Z_2^n$-context. }
\end{abstract}

\vspace{2mm} \noindent {\bf MSC 2010}: 17A70, 58A50, 13F25, 16L30 \medskip


\thispagestyle{empty}
\tableofcontents

\section{Introduction}

Classical Supersymmetry and Supergeometry are not sufficient to suit the current needs. {In {\it Physics}, $\Z_2^n$-gradings, $n\ge 2,$ are used in string theory and in parastatistics \cite{AFT10}, \cite{YJ01}.} In {\it Mathematics}, there exist good examples of $\Z_2^n$-graded {\it $\Z_2^n$-commutative algebras} (i.e. the superscript of $-1$ in the sign rule is the standard `scalar product' of $\Z_2^n$): quaternions and, more generally, any Clifford algebra, the algebra of Deligne differential superforms... And there exist interesting examples of {\it $\Z^n_2$-supermanifolds}: {e.g., tangent and cotangent bundles $\sT M$ and $\sT^\star M$, double vector bundles such as $\sT\sT M$, and, more generally, $n$-vector bundles...}\medskip

Indeed, the tangent bundle of a classical $\Z_2^1$-supermanifold $\cM$ is a $\Z_2^1$-supermanifold $\sT[1]\cM$ (resp., a $\Z_2^2$-supermanifold $\sT\cM$) with function sheaf the differential superforms of $\cM$ together with the Bernstein-Leites (resp., with the Deligne) sign convention. Actually the tangent (and cotangent) bundle(s) of any $\Z_2^n$-supermanifold is a (are) $\Z_2^{n+1}$-supermanifold(s). Further, any $n$-vector bundle canonically provides a $\Z_2^n$-supermanifold {as its `superization'}.\medskip

To be more precise, suppose that $\cM$ is a supermanifold with local coordinates $(x^1,\dots,x^p,$ $\zx^1,\dots,\zx^q)$, where the $x^i$ are even and the $\zx^a$ odd. For the tangent bundle $\sT\cM$, with the adapted local coordinates $(x^i,\zx^a,\dot x^j,\dot\zx^b)$, one can introduce a supermanifold structure, in principle, in two ways: declaring $\dot x^j$ to be even and $\dot\zx^b$ to be odd, or reversing these parities.\medskip

For the latter, {the variables $\dot\zx^b$ are even; they are true real-valued variables although the $\zx^b$ are indeterminates or formal variables. Indeed, the $\dot\zx^b$ can for instance be thought of as the velocities of a spinning particle. These can actually be measured, whereas the $\xi^b$ cannot. Hence, the local model of the function algebra of the corresponding supermanifold $\sT[1]\cM$ is given by the polynomials in the odd indeterminates $\zx^a$ and $\dot x^j$ with coefficients in the smooth functions with respect to the even variables $x^i$ and $\dot \zx^b$. These polynomials are referred to as pseudodifferential forms. Pseudodifferential forms can be defined not only locally on superdomains but also globally on supermanifolds \cite{Lei11}. They have been introduced in \cite{BL77} to obtain objects suitable for integration. The algebra $\widehat{\zW}(\cM)$ of pseudodifferential forms on $\cM$ is exactly the function algebra of the supermanifold $\sT[1]\cM$.\medskip}

On the other hand, $\sT\cM$, as every vector bundle, admits an $\N$-grading for which $\dot x^j$ and $\dot\zx^b$ are of degree 1 {(and $x^i$ and $\zx^a$ are of degree 0)}. Thus we have a canonical bigrading by the monoid $\N\times \mZ$, which can be reduced to $\mZ^2=\mZ\ti\mZ$. With respect to this bigrading, $(x^i,\zx^a,\dot x^j,\dot\zx^b)$ are of bidegrees $(0,0),(0,1),(1,0)$, and $(1,1)$, respectively. Now, any symmetric biadditive map $\la -,- \ran:\mZ^2\ti\mZ^2\to\Z$ gives rise to a sign rule: $$AB=(-1)^{\la(m,n),(k,l)\ran}BA\;,$$ where $A$ and $B$ are coordinates of bidegrees $(m,n)$ and $(k,l)$, respectively. We get the usual sign rule when choosing $\la(m,n),(k,l)\ran=mk$, and obtain the sign rule for reversed parity (Berntein-Leites sign rule) choosing
$\la(m,n),(k,l)\ran=(m+n)(k+l)$, whereas the `scalar product' $\la(m,n),(k,l)\ran=mk+nl$ has been used by Deligne -- see discussion in \cite[Appendix to \S 1]{DM99}. Note that the latter does not lead to a superalgebra, as the $\dot\zx^b$ are even in the sense that they commute among themselves, but anticommute with the $\zx^a$. Not excluding any sign rule forces us to work with this bigrading and to include Deligne's convention into the picture, which -- as mentioned -- does not correspond to any supermanifold. It is therefore natural to extend the notion of supermanifold admitting $\mZ^n$-gradings and the corresponding sign rule \be\label{Z_2^n-com}AB=(-1)^{\sum_{i=1}^n m_ik_i}BA\;\ee (here $m=(m_1,\ldots,m_n)$ and $k=(k_1,\ldots,k_n)$ are the degrees of $A$ and $B$, respectively), so that the additional canonical gradings on $\sT\cM$ or $\sT^\*\cM$ do not move us out of the corresponding category.\medskip

Although not universally accepted at the beginning, $\Z_2^n$-Supergeometry is thus a necessary and natural generalization. When defining the parity of a $\Z_2^n$-degree as the parity of the total degree, nonzero degree even coordinates are not nilpotent, and even (resp., odd) coordinates do not necessarily commute (resp., anticommute) pairwise. These circumstances lead to challenging {differences} with the classical theory. The reason for initial skepticism was Neklyudova's equivalence \cite{Lei11}: this result states that the categories of $\Z_2^n$-graded {\it $\Z_2^n$-commutative} and $\Z^n_2$-graded {\it supercommutative} algebras are equivalent. However, our previous work shows that pullbacks of `{\it supercommutative} concepts' to the {\it $\Z_2^n$-commutative} setting are not always as easy as expected, and do not always operate properly: Neklyudova's theorem {\bf does not ban} studies of $\Z_2^n$-commutative algebras! On the other hand, $\Z_2^n$-commutative algebras {\bf are sufficient}, in the sense that any sign rule, for any finite number $m$ of coordinates, is of the form (\ref{Z_2^n-com}), for some $n\le 2m$.\medskip

Actually $\Z_2^n$-Supergeometry focusses on the following hierarchy: classical differential Geometry ($\Z_2^0$-Supergeometry) contains the germ (differential forms) of standard Supergeometry ($\Z_2^1$-Supergeometry), which in turn contains the sprout (Deligne superforms) of $\Z_2^2$-Supergeometry... The $\Z_2^n$-supergeometric viewpoint provides improved insight and simplified solutions, thus emphasizing an effect already observed for classical Supergeometry. Tight relations with diverse concepts in Quantum Field Theory and Quantum Mechanics can be expected. The theory of $\Z_2^n$-supermanifolds is closely related to Clifford calculus. Clifford algebras have numerous applications in Physics, {but the use of $\Z_2^n$-gradings has never been studied. For instance, Clifford algebras and modules are crucial in understanding Spin-structures on manifolds together with their physical consequences (e.g., the Dirac operator); while the $\Z_2^n$-refined viewpoint led to new results on Clifford algebras and modules over them [COP12], its impact on standard applications in Geometry and Field Theory has still to be explained -- the results of this project are currently being written down in a separate paper.} Further, our theory should lead to a novel approach to quaternionic functions: examples of application areas include thermodynamics, hydrodynamics, geophysics and structural mechanics. It is further interesting to observe the parallelism of our extension with Baez' suggestion of a common generalization -- under the name of $r$-Geometry -- of superalgebras and Clifford algebras with the goal to incorporate, besides bosons and fermions, also anyons into the picture \cite{Bae92}.\medskip

Finally, the key-concept of $\Z_2^n$-Superalgebra is the $\Z_2^n$-Berezinian. This higher Berezinian (which is tightly connected with quasi-determinants) and the corresponding (via the Liouville formula) higher trace have recently been constructed \cite{COP12}. It provides a new solution (`different' from the Dieudonn\'e determinant) to Cayley's challenge to build a determinant of quaternionic matrices. Hence, our $\Z_2^n$-Supergeometry not only includes differential but also integral calculus, whereas Molotkov, who developped a (functorial) concept of $\Z_2^n$-supermanifold, mentions explicitly that he has no insight in this respect \cite{Mol10}.\medskip

The paper is organized as follows. In Section 2, we prove that any sign rule, for any finite set of variables, is of the type (\ref{Z_2^n-com}), for some $\Z_2^n$-grading. The necessity to consider $\Z_2^n$-superdomains, characterized just as in \cite{Mol10} by algebras of formal power series, is explained in Section 3. Moreover, invertibility, locality and completeness issues are addressed, and a coordinate version of the $\Z_2^n$-supermorphism theorem is proved. In Sections 2 and 3, we use exclusively coordinate computations, thus allowing the reader to get acquainted with the specifics and foundations of the new theory. The latter is developed in the next sections via the atlas approach, as well as, mainly, in the ringed space setting. The concept of $\Z_2^n$-supermanifold is introduced in Section 4. In Section 5, we detail first examples. Section 6 contains the main proofs of the work at hand. We show that most important results of classical Supergeometry extend to the $\Z_2^n$-context, although nilpotency is lost in this generalized framework -- it turns out that formal series are the appropriate substitute. We prove that $\Z_2^n$-superfunctions project consistently to the base and that the latter actually carries a smooth manifold structure. Continuity of the pullback maps of morphisms between $\Z_2^n$-supermanifolds with respect to the filtration provided by the kernel of the base projection, as well as continuity of the induced maps between stalks with respect to the filtration implemented by the unique maximal homogeneous ideal -- combined with an appropriate polynomial approximation of $\Z_2^n$-superfunctions --, allow us to show that the fundamental theorem of supermorphisms extends to the $\Z_2^n$-setting. Complementary information can be found in the appendix-section 7.\medskip

Let us finally provide a non-exhaustive list of references on classical supermanifolds and related topics that were of importance for the present text: \cite{DAL}, \cite{Lei11}, \cite{Var}, \cite{Man}, \cite{DM99}, \cite{CCF}, \cite{DSB}, \cite{Vor12}, \cite{CR12}, \cite{BP12}, \cite{GKP1}, \cite{GKP2}.

\section{Sign rules}

Supergeometry is commonly understood as the theory of manifold-like objects admitting anticommuting variables. This corresponds to a $\mZ$-grading in the structure sheaf of the corresponding ringed space, so that even elements are central (commute with everything) and odd elements anticommute among themselves. In particular, they are 2-nilpotent.
This means that the sign rules between generators of the algebra are completely determined by their \emph{parity}. Why not accept arbitrary commutation rules between different generators, even with fixed parities?
In principle, one can consider a general grading by a semigroup and an arbitrary commutation factor, i.e. work with so-called \emph{colored algebras}.\medskip

More precisely, let $K$ be a commutative unital ring, $K^\times$ be the group of invertible elements
of $K$, and let $G$ be a commutative semigroup. A map $\ze: G \ti G \to K^\times$ is called a \emph{commutation factor} on $G$ if
$$\ze(g,h)\ze(h,g) = 1\,,\quad\ze(f,g+h)=\ze(f,g)\ze(f,h)\,,\quad \text{and}\quad \ze(g,g) = \pm 1\,,$$
for all $f, g, h\in G$.
Note that these axioms imply that
$$\ze(f+g,h)=\ze(f,h)\ze(g,h)$$
and that the condition $\ze(g,g) = \pm 1$ follows automatically from the other two axioms if $K$ is a field.\medskip

Let $\cA$ be a $G$-graded $K$-algebra $\cA=\bigoplus_{g\in G}\cA^g$. Elements $x$ from $\cA^g$ are called \emph{$G$-homogeneous
of degree} or \emph{weight} $g=:\op{deg}(x)$. The algebra $\cA$ is said to be \emph{$\ze$-commutative} if
$$ a b = \ze(\op{deg}(a),\op{deg}(b))b a\;,$$
for all $G$-homogeneous elements $a, b \in \cA$. Homogeneous elements $x$ with $p(\op{deg}(x))=p(g):=\ze(g,g)=-1$
are \emph{odd}, the other homogeneous elements are \emph{even}.\medskip

In what follows, $K$ will be $\R$ and $\ze$ will take the form $$\ze(g,h)=(-1)^{\la g, h\ran}\;,$$ for
a `scalar product' $\la -,- \ran:G\ti G\to\Z$. This means that we use the {\it commutation factor} as the \emph{sign rule}. In this note we confine ourselves to $G=\mZ^n$ and the standard `scalar product' of $\Z_2^n$, what will lead to $\Z_2^n$-Supergeometry with nicer categorical properties than the standard Supergeometry. More precisely, we propose a generalization of differential $\mZ^1$-Supergeometry to the case of a $\mZ^n$-grading in the structure sheaf.\medskip

Indeed, we will show that any sign rule, for any finite number of coordinates, can be obtained from the `scalar product'
$$\la(i_1,\dots,i_n),(j_1,\dots,j_n)\ran_n=i_1j_1+\ldots +i_nj_n
$$ on $\mZ^n$ for a sufficiently big $n$. In other words, any algebra that is finitely generated by some generators satisfying certain sign rules can be viewed as a $\mZ^n$-graded associative algebra $\cA=\bigoplus_{i\in\mZ^n}\cA^i$, which is $\Z_2^n$-commutative in the sense that
$$y^iy^j=(-1)^{\la i,j\ran_n}y^jy^i\;,$$
for all $y^i\in\cA^i$, $y^j\in\cA^j$. We simply refer to such algebras as {$\mZ^n$-\emph{commutative associative algebras}}. Let us mention that a similar theorem was proved in \cite{MGO2} {for group gradings}.\medskip

Let now $S$ be a finite set, say $S=\{ 1,\dots,m\}$, and let $\ze:S\ti S\to\{\pm 1\}$ be any symmetric function. We can understand $\ze$ as a sign rule for an associative algebra generated by elements $y^i$, $i=1,\dots, m$, i.e.
$$y^iy^j=\ze(i,j)y^jy^i\,.$$

We then have the
\begin{thm}\label{SuffGrad} There is $n\le 2m$ and a map $\zs:S\to\mZ^n$, $i\mapsto\zs_i$, such that \be\label{sr}\ze(i,j)=(-1)^{\la\zs_i,\zs_j\ran_n}\;.
\ee
\end{thm}
\begin{proof}
We interpret $\mZ^{2m}$ as the set of functions $\{\pm 1,\dots,\pm m\}\to\mZ$, and denote by $p(i,j)\in\{0,1\}$ the \emph{parity} of $\ze(i,j)$: $(-1)^{p(i,j)}=\ze(i,j)$.

First, define $\zs_1\in \mZ^{2m}$ by $\zs_1(1)=1$, $\zs_1(-1)=1+p(1,1)\in\mZ$, and $\zs_1(k)=0$ for $|k|>1$. Then,  for $j=2,\dots,m$, define $\zs_j(1)=p(j,1)$ and $\zs_j(-1)=0$. Independently of the definition of the remaining values of $\zs_j$, Condition (\ref{sr}) is valid for $i=1$ and all $j=1,\dots, m$, since $\zs_1(k)=0$ for $|k|>1$.

Assume inductively that we have fixed $\zs_1,\dots,\zs_r$, with $\zs_j(k)=0$ for $|k|>j$, as well as the values $\zs_j(k)$, for $j=r+1,\dots, m$ and $|k|\le r$, so that (\ref{sr}) is valid for $i=1,\dots,r$ and all $j$.

Define:
$$\zs_{r+1}(r+1)=1\,, \quad\zs_{r+1}(-r-1)=1+\sum_{|k|=1}^r\zs_{r+1}(k)+p(r+1,r+1)\,,\quad \zs_{r+1}(k)=0\ \text{for}\ |k|>r+1\,.$$
Then, (\ref{sr}) is valid also for $i=j=r+1$. Putting now $\zs_j(-r-1)=0$ and
$$\zs_j(r+1)=\sum_{|k|=1}^r\zs_j(k)\zs_{r+1}(k)+p(j,r+1)$$
for $j=r+2,\dots, m$, we finish with fixed $\zs_1,\dots,\zs_{r+1}$, with $\zs_j(k)=0$ for $|k|>j$, and the values $\zs_j(k)$, for $j=r+2,\dots, m$ and $|k|\le r+1$, so that (\ref{sr}) is valid for $i=1,\dots,r+1$ and all $j$. This proves the inductive step and the theorem follows.
\end{proof}

Let now $y^1,\dots,y^m$ be `variables' with $\mZ^n$-degrees fixed by a map $\zs:\{ 1,\dots,m\}\to\mZ^n$. We can consider $\R[y^1,\dots,y^m]_\zs$, which is the free graded tensor algebra over reals generated by variables $y^1,\dots,y^m$ modulo the commutation relations described by $\zs$,
$$
y^iy^j=(-1)^{\la \zs_i,\zs_j\ran_n}y^jy^i\;
$$
\cite{CM14}. This algebra is referred to as the \emph{free $\zs$-commutative associative $\R$-algebra}, or, if $\zs$ is fixed, the \emph{free $\Z_2^n$-commutative associative $\R$-algebra} in $m$ generators. Moreover, if $n$ is fixed, we usually omit subscript $n$ in $\la-,-\ra_n$. The variable $y^i$ is even (resp., odd) if $$p(y^i):=|\zs_i|:=\zs_i(1)+\ldots+\zs_i(n)\in\mZ$$ is 0 (resp., 1). We can write every element of $\R[y^1,\dots,y^m]_\zs$ uniquely as a polynomial
$$
 f(y)=\sum_{|\mu|=0}^{N_f}
  f_{\mu_1\ldots\mu_m}(y^1)^{\mu_1}\ldots(y^m)^{\mu_m}
 =\sum_{|\mu|=0}^{N_f} f_\mu y^{\,\mu}\,,$$
where $|\mu|=\zm_1+\ldots+\zm_m$. %It is clear that this algebra has the following universal property.
%\begin{prop}\label{p1} For any $\mZ^n$-commutative associative unital $\R$-algebra $\A$ and any degree-preserving map
%$$\underline{\varphi}:\{y^1,\dots,y^m\}\to\A\,,\quad \underline{\varphi}(y^i)\in\A^{\zs_i}\,,$$ there is a unique graded unital $\R$-algebra morphism $\varphi:\R[y^1,\dots,y^m]_\zs\to\A$ extending
%$\underline{\varphi}$.
%\end{prop}

\section{$\mathbb{Z}_2^n$-superdomains and their morphisms}

To develop a generalization of Supergeometry, we wish to distinguish coordinates $x^1,\dots, x^p$ of degree $0:=(0,\ldots,0)\in\Z_2^n$ and view them as local coordinates on a standard manifold. The remaining coordinates $\zx^1,\dots,\zx^q$ have nontrivial degrees $\zs_1,\dots,\zs_q\in\Z_2^n\setminus\{ 0\}$ determined by a fixed map $\zs:\{1,\dots,q\}\to\Z_2^n\setminus\{ 0\}$. We will call them \emph{indeterminates} or \emph{formal variables}.

\subsection{Sheaf of polynomials}

The first idea is to define a $\zs$-superdomain or $\Z_2^n$-superdomain as a \emph{ringed space}
$\frak{U} = (U,\frak{O}_{U,\zs})$, where $U\subset\R^p$ is an open subset and the structure sheaf is given by
$$
 \frak{O}_{U,\zs}(-):=
 C^\infty_U(-)[\xi^1,\dots,\xi^q]_\zs\;.
$$
Here $\xi^1,\dots,\xi^q$ is a sequence of variables of $\Z_2^n$-degrees $\zs_a$, i.e. commuting according to
\be\label{comru}
\zx^a\zx^b=(-1)^{\la\zs_a,\zs_b\ran_n}\zx^b\zx^a\;.
\ee
As already mentioned above, we omit in the sequel subscript $\zs,$ since this map is fixed. Thus, on $V\subset U$, our algebra of superfunctions would be the $\Z_2^n$-commutative associative unital $\R$-algebra
$$
\frak{O}_U(V)=C^\infty_U(V)[\zx^1,\dots,\zx^q]
$$
of polynomials
$$
 f(x,\zx)=\sum_{|\mu|=0}^{N_f}
  f_{\mu_1\ldots\mu_q}(x)\;(\zx^1)^{\mu_1}\ldots(\zx^q)^{\mu_q}
 =\sum_{|\mu|=0}^{N_f} f_\mu(x) \zx^\mu\,
$$
in the variables $\zx^a$ and with coefficients in the ring $\Ci(V)$, whose multiplication is subject to the sign rules determined by (\ref{comru}). Note that we omit subscripts like $U$, whenever we do not wish to stress the (Hausdorff, second-countable) topological space over which the considered sheaf is defined. Of course, those $\zx^a$ which are odd, $p(\zx^a)=1$, appear in the polynomials with exponents $\le 1$.\medskip

Morphisms $\frak{O}(W)\to \frak{O}(V)$ of $\Z_2^n$-commutative associative unital $\R$-algebras (in particular changes of coordinates) should preserve the grading, so have the form
\bea\label{mor}
x'^i&=&\zf^i(x)+\sum_{\op{deg}(\xi^\mu)=0} f^i_\mu(x) \zx^\mu\,,\\
\zx'^a&=&\sum_{\op{deg}(\xi^\mu)=\zs_a} f^a_\mu(x) \zx^\mu\,,\nn
\eea
where the functions $f_\zm:V\to\R^p$ and the map $\zf:V\to W$ are smooth, and the sums are finite.\medskip

It is easy to see that the ideal $\frak{J}(V)\subset \frak{O}(V)$ generated by the formal variables is respected by morphisms and that the projection
$$\frak{p}_V:\frak{O}(V)\to \frak{O}(V)/\frak{J}(V)\simeq \Ci(V)$$
is covariantly defined (we come back to this and similar points later on).

However, this approach has clear shortcomings.

First, as we allow formal variables which are even, the ideal $\frak{J}(V)$ is not nilpotent, in general, so superfunctions $f$ with invertible `body' $\frak{p}_V(f)$ need not to be invertible in the ring $\frak{O}(V)$. Formal inverting of polynomials requires using formal power series.

Second, for a proper development of differential calculus, we should be able to compose elements of degree $0$, see (\ref{mor}), with arbitrary differentiable functions and not only polynomials. But what is $F(x+\zx^2)$ for a differentiable $F$ and formal  even variable $\zx$? Since $\xi$ is not nilpotent, the Taylor formula (proceed as in standard Supergeometry) leads again to a formal power series.

\subsection{Sheaf of formal power series}

{A consistent differential calculus for $\Z_2^n$-superdomains forces us to complete the structure sheaf to formal power series in the indeterminates. This local model is the same than the one obtained by Molotkov \cite{Mol10} via his functorial approach to higher graded supermanifolds. It should be noticed that, when the $\Z_2^n$ sign rule (\ref{comru}) replaces the classical super sign rule, even indeterminates may anticommute with even and with odd indeterminates (e.g., if they are $\Z_2^3$-graded and have the degrees $(1,1,0)$, $(0,1,1)$ and $(0,1,0)$, respectively). On the other hand, the algebra of quaternions $$\mathbb{H}=\R\oplus \R i\oplus \R j \oplus \R k$$ is $\Z_2^3$-commutative, if we choose the degrees $\deg 1=(0,0,0)$, $\deg i=(1,1,0)$, $\deg j=(1,0,1)$, and $\deg k=(0,1,1)$. All this shows that the even non-zero degree indeterminates are not usual even variables and that we should think about them as formal variables rather than as ordinary ones. Hence, the local $\Z_2^n$-function algebra will be made of formal power series in the odd and the non-zero degree even indeterminates. Indeed, what would for instance be the definition of a smooth function with respect to $i,j,k$? What would be the meaning of the sine $\sin(a+bi+cj+dk)$ of a quaternion?}\medskip

{It follows that in the case of the tangent bundle to a supermanifold $\cM$, the local functions of the supermanifold $\sT[1]\cM$ and of the $\Z_2^2$-manifold $\sT\cM$ are quite different. As mentioned in the introduction, if $(x^i,\zx^a)$ are even and odd coordinates of $\cM$, the supermanifold $\sT[1]\cM$ has coordinates $(x^i,\zx^a,\dot x^j,\dot \zx^b)$ of parities $(0,1,1,0)$ and subject to the super commutation rule; the superfunctions of $\sT[1]\cM$ are thus locally the polynomials in $\zx^a,\dot x^j$ with coefficients in the smooth functions with respect to $x^i,\dot \zx^b$. The case of the $\Z_2^2$-manifold $\sT\cM$, with coordinates $(x^i,\zx^a,\dot x^j,\dot \zx^b)$ of bidegrees $((0,0),(0,1),(1,0),(1,1))$ and subject to the $\Z_2^2$-commutation rule, is different. Its $\Z_2^2$-functions are locally the formal power series in the non-zero degree formal variables $\zx^a,\dot x^j, \dot \zx^b$ with coefficients in the smooth functions in $x^i$.}\medskip

{It follows that the base of a $\Z_2^n$-manifold corresponds, not to the even variables but to the zero-degree ones. In this sense $\Z_2^n$-supermanifolds are similar to $\Z$-graded manifolds. Indeed, the base of a $\Z$-graded manifold $\cM$ can be recovered as the spectrum of $C^0(\cM)/(I\cap C^0(\cM))$, where $C^0(\cM)$ are the zero-degree functions of $\cM$ and where $I$ is the ideal generated by the coordinates of non-zero degree \cite{Roy02}.}\medskip

In what follows, we consider the  $n$-tuples of $\Z_2^n$ as ordered lexicographically.

\begin{definition} Let $n,p,q\in\N$ and let $\zs:\{1,\ldots,q\}\to\Z_2^n\setminus\{0\}$. Denote by $q_k\in\N$, where $k\in\{1,\ldots,2^n-1\},$ the number of degrees $\zs_a$ that coincide with the $k$-th element of $\Z_2^n\setminus\{0\}$ and set $\mathbf{q}=(q_1,\ldots,q_{2^n-1})$. A \emph{$\zs$-superdomain} or \emph{$\Z_2^n$-superdomain} of dimension $p|\mathbf{q}$ is a \emph{ringed space}
${\cU^{\,p|\mathbf{q}}} = (U,{\cO}_{U,\zs})$, where $U\subset\R^p$ is an open subset and the structure sheaf is the sheaf
$$
 {\cO}_{U,\zs}(-):=
 C^\infty_U(-)[[\xi^1,\dots,\xi^q]]_\zs\;.
$$
\end{definition}

Over $V\subset U,$ the algebra of $\Z_2^n$-functions is the $\Z_2^n$-commutative associative unital $\R$-algebra
$$
{\cO}_U(V)=C^\infty_U(V)[[\zx^1,\dots,\zx^q]]
$$
of formal power series
$$
 f(x,\zx)=\sum_{|\mu|=0}^{\infty}
  f_{\mu_1\ldots\mu_q}(x)(\zx^1)^{\mu_1}\ldots(\zx^q)^{\mu_q}
 =\sum_{|\mu|=0}^{\infty} f_\mu(x) \zx^\mu\;
$$
in formal variables $\xi^1,\dots,\xi^q$ of degrees $\zs_1,\ldots,\zs_q$ commuting according to
(\ref{comru}), and with coefficients in $\Ci(V)$.\medskip

We refer to a ringed space of $\Z_2^n$-commutative associative unital $\R$-algebras as a \emph{$\Z_2^n$-ringed space}.

\begin{example}\label{FundaExa} Consider the case $n=2$ and $p|q_1|q_2|q_3=1|1|1|1$, write for simplicity $(x,\xi,\zh,\zy)$ instead of $(x,\xi^1,\xi^2,\xi^3)$, and choose $\zs_\xi=(0,1),\zs_\zh=(1,0)$, and $\zs_\zy=(1,1)$. A $\Z_2^2$-function is then of the form 
\beas\label{Superfunction}
f(x,\xi,\zh,\zy)&=&\sum_{r\ge 0} f_r(x)\zy^{2r}+\sum_{r\ge 0} g_r(x)\zy^{2r+1}\xi\zh+\sum_{r\ge 0} h_r(x)\zy^{2r}\xi+\sum_{r\ge 0} k_r(x)\zy^{2r+1}\zh \\
&& +\sum_{r\ge 0} \ell_r(x)\zy^{2r}\zh+\sum_{r\ge 0} m_r(x)\zy^{2r+1}\xi+\sum_{r\ge 0} n_r(x)\zy^{2r+1}+\sum_{r\ge 0} p_r(x)\zy^{2r}\xi\zh\;,
\eeas
where the sums are formal series and the functions in $x$ are smooth. Note that the first (resp., second, third, fourth) two sums contain terms of $\Z_2^2$-degree $(0,0)$ (resp., $(0,1)$, $(1,0)$, and $(1,1)$).\end{example}

\subsection{Locality of $\Z_2^n$-superdomains}

In classical Supergeometry a (super) ringed space is called a \emph{space} if all its stalks are local rings, i.e. rings that have a unique maximal homogeneous ideal. Such ringed spaces are referred to as \emph{locally ringed spaces}. Further, a ringed space is a \emph{supermanifold} if it is a space that is locally modelled on a superdomain. Superdomains are thus `trivial' locally ringed spaces. Of course, one has to verify that the stalks of a superdomain are local rings.\medskip

To show that $\Z_2^n$-superdomains are \emph{locally $\Z_2 ^n$-ringed spaces}, we need two lemmas.\medskip

Let $R$ be a commutative unital ring and let $(\xi^1,\ldots,\xi^q)$ be a finite number of $\Z_2^n\setminus\{0\}$-graded parameters, which satisfy $$\xi^i\xi^j=(-1)^{\la \op{deg}(\xi^i),\op{deg}(\xi^j)\ra}\xi^j\xi^i$$ (the scalars $R$ are assumed to be central). We denote by $R[[\xi^1,\ldots,\xi^q]]$ the $\Z_2^n$-commutative associative unital $R$-algebra of formal series in the $\xi^a$ with coefficients in $R$.

\begin{lem} Any series $1-v$, where $v=\sum_{|\zm|>0}v_{\zm}\xi^{\zm}$ has no independent term, is invertible, with inverse $v^{-1}=\sum_{k\geq0} v^k$.\end{lem}

\begin{proof} Observe first that, for any $k\in\N$,
$$
v^k= \sum_{|\zn|\geq k} \lp \sum_{\zm_1+ \ldots + \zm_k=\zn}\pm v_{\zm_1}\cdots v_{\zm_k}  \rp \xi^{\zn}\;,
$$
where the $\zm_i\in \N^q$ are of course multi-indices. It follows that the coefficients of $v^{-1}:=\sum_{k\geq0} v^k$ are finite sums in $R$, so that $v^{-1}\in R[[\xi^1,\ldots,\xi^q]]$. It suffices now to observe that

$$
(1-v)\sum_{k\geq0} v^k = \sum_{k\geq0} v^k - \sum_{k\geq1} v^k = 1\;.
$$
\end{proof}

\begin{lem}\label{lem:FPSinvert}
A series $w\in R[[\xi^1,\ldots,\xi^q]]$ is invertible if and only if its independent term $w_0$ is invertible in $R$.
\end{lem}

\begin{proof}
Necessity directly follows from the definition of the multiplication in $R[[\xi^1,\ldots,\xi^q]]$. Conversely, consider $w\in R[[\xi^1,\ldots,\xi^q]]$ with $w_0$ invertible: $w=w_0(1-v)$. In view of the preceding lemma, we then have $w^{-1}= w_0^{-1}\sum_{k\geq0} v^k$.
\end{proof}

We are now prepared to prove the

\begin{prop}\label{prop:domlocality} Any $\Z_2^n$-superdomain $(U,C^\infty_U[[\xi^1,\dots,\xi^q]])$ is a locally $\Z_2^n$-ringed space, i.e. for any $x\in U$, the stalk $\Ci_{U,x}[[\xi^1,\ldots,\xi^q]]$ has a unique maximal homogeneous ideal
$$
\frak{m}_x=\{[f]_x: f_0(x)=0\}\;.
$$
\end{prop}

\begin{proof} Set $S_x=\Ci_{U,x}[[\xi^1,\ldots,\xi^q]]$. In view of Lemma \ref{lem:FPSinvert}, a series $[f]_x\in S_x$ is invertible if and only if $[f_0]_x\in\Ci_{U,x}$ is invertible, i.e. if and only if $f_0(x)\neq 0$:
$$
S_x\setminus S_x^{\times} =\{ [f]_x: f_0(x)=0 \}\;.
$$
The latter is clearly a proper homogeneous ideal. Let $I_x$ be any proper homogeneous ideal. If it strictly contains $S_x\setminus S_x^\times$, it contains an invertible element of $S_x$ and can thus not be proper: the homogeneous ideal $\frak{m}_x:=S_x\setminus S_x^\times$ is maximal. If $I_x$ is another maximal homogeneous ideal, it does not contain any invertible element: $I_x \subset \frak{m}_x\subset S_x$ -- a contradiction.\end{proof}

Moreover, Lemma \ref{lem:FPSinvert} has the following

\begin{cor}\label{p1} For any open $V\subset U$, a $\Z_2^n$-function $f\in {\cO}_U(V)=\Ci_U(V)[[\xi^1,\ldots,\xi^q]]$ is invertible in ${\cO}_U(V)$ if and only if its independent term $f_0$ is invertible in $\Ci_U(V)$.
\end{cor}

This corollary guarantees that a number of results of classical Supergeometry still hold in $\Z_2^n$-Supergeometry, although formal variables are no longer necessarily nilpotent.

\subsection{Completeness of $\Z_2^n$-function algebras}

The algebra $\cO(V)=\Ci(V)[[\xi^1,\ldots,\xi^q]]$ of formal power series is the completion of the algebra $\frak{O}(V)=C^\infty(V)[\zx^1,\dots,\zx^q]$ of polynomials. Moreover,

\begin{prop} The algebra $\cO(V)=\Ci(V)[[\xi^1,\ldots,\xi^q]]$ of $\Z_2^n$-functions on $V$ is Hausdorff-complete (in the sense of Section \ref{App1}).\end{prop}

\begin{proof} Consider a $\Z_2^n$-superdomain with $\Z_2^n$-functions 
$$f(x,\xi)=\sum_{|\mu|=0}^{\infty} f_\mu(x) \zx^\mu\in \cO(V)=\Ci(V)[[\xi^1,\ldots,\xi^q]]\;.$$ 
The number $k:=|\zm|$ of generators defines an $\N$-grading in $\cO(V)$ that induces a decreasing filtration 
$\cO_\ell(V)=\Ci(V)[[\xi^1,\ldots,\xi^q]]_{\ge\ell}$, 
where subscript $\ge \ell$ means that we consider only series whose terms contain at least $\ell$ parameters $\xi^a$ 
(in the following we omit $V$ if no confusion arises). 
Of course $J=J^1:=\cO_1$ -- the kernel of the projection of $\Z_2^n$- onto base-functions -- is an ideal of $\cO$ 
and $J^\ell=\cO_\ell$: $$\cO\supset J\supset J^2\supset\ldots$$ 
The sequence $\cO/J\leftarrow \cO/J^2\leftarrow\cO/J^3\leftarrow\ldots\;$, which can be identified with the sequence 
$$\Ci\leftarrow \Ci[[\xi^1,\ldots,\xi^q]]_{\le 1}\leftarrow \Ci(V)[[\xi^1,\ldots,\xi^q]]_{\le 2}\leftarrow\ldots\;,$$ 
is an inverse system, whose limit is 
\be\label{Completeness}\varprojlim_\ell \cO/J^\ell=\cO\;.\ee 
This means that $\cO$ is Hausdorff-complete, see Section \ref{App1}.
\end{proof}

\begin{itemize}\item It is well known that Equation (\ref{Completeness}) means that $\cO$ is a complete topological algebra with respect to the topology in $\cO$ defined by the filtration $J^\ell,$ $\ell\ge 1$, viewed as a basis of neighborhoods of $0$. \item Remark that also the sequence $\frak{O}/\frak{J}\leftarrow \frak{O}/\frak{J}^2\leftarrow\frak{O}/\frak{J}^3\leftarrow\ldots\;$ can be identified with $\Ci\leftarrow \Ci[[\xi^1,\ldots,\xi^q]]_{\le 1}\leftarrow \Ci(V)[[\xi^1,\ldots,\xi^q]]_{\le 2}\leftarrow\ldots\;$ It follows that \be\label{Completeness2}\varprojlim_\ell \frak{O}/\frak{J}^\ell=\cO\;,\ee so that $\cO$ is actually the completion $\widehat{\frak O}$ of $\frak{O}$ with respect to the filtration implemented by $\frak J$ (as well as, in view of (\ref{Completeness}), its own completion with respect to $J$).\end{itemize}

\subsection{Morphisms of $\Z_2^n$-superdomains}\label{MorphTheo}

The following remark shows that morphisms of $\Z_2^n$-superdomains can be viewed as in classical differential Geometry. It will be formulated more rigorously in the case of general $\Z_2^n$-supermanifolds.\medskip

Consider two $\Z_2^n$-superdomains of dimension $p|\mathbf{q}$ and $p'|\mathbf{q}'$ over open subsets $U\subset\R^p$ and $U'\subset\R^{p'}$, respectively. Roughly, $\Z_2^n$-morphisms between these $\Z_2^n$-superdomains correspond to graded unital $\R$-algebra morphisms
$$\phi^*:\Ci(V')[[\zx'^1,\dots,\zx'^{q'}]]\to \Ci(V)[[\zx^1,\dots,\zx^{q}]]\;$$
and are determined by their coordinate form
\bea\label{mor1}
x'^i&=&\zf^i(x)+\sum_{\zs(\mu)=0} f^i_\mu(x) \zx^\mu\,,\\
\zx'^a&=&\sum_{\zs(\mu)=\zs_a} f^a_\mu(x) \zx^\mu\,,\nn
\eea
where the sums are formal series with coefficients in smooth functions and where $\zf:V\ni (x^1,\ldots,x^p)\mapsto (x'^1,\ldots,x'^{p'})\in V'$ is a smooth map.\medskip

\begin{example}\label{FundaExa2} 
In the case of $\Z_2^2$-superdomains of dimension $1|1|1|1$ with variables $(x,\xi,\zh,\zy)$ (resp., $(y,\za,\zb,\zg)$) of $\Z_2^2$-degrees $((0,0),(0,1),(1,0),(1,1))$, a $\Z_2^2$-morphism can be viewed as usual: 
$$\label{Morph}\left\{\begin{array}{c} y=\sum_r f^y_r(x)\zy^{2r}+\sum_r g^y_r(x)\zy^{2r+1}\xi\zh\;,\\ \za=\sum_r f^\za_r(x)\zy^{2r}\xi+\sum_r g^\za_r(x)\zy^{2r+1}\zh\;, \\ \zb=\sum_r f^\zb_r(x)\zy^{2r}\zh+\sum_r g^\zb_r(x)\zy^{2r+1}\xi\;, \\ \zg=\sum_r f^\zg_r(x)\zy^{2r+1}+\sum_r g^\zg_r(x)\zy^{2r}\xi\zh\;.\end{array}\right.\;$$
\end{example}\smallskip

To explain the above claim, we have to prove that any $\Z_2^n$-morphism has a coordinate form of the announced type (what is almost obvious), and that, conversely, any pullbacks $\phi^*(x'^i)\;(\simeq x'^i)$ and $\phi^*(\xi'^a)\;(\simeq \xi'^a)$ of the form (\ref{mor1}) uniquely extend to a $\Z_2^n$-morphism. We will show here that such a $\Z_2^n$-morphism does exist. Uniqueness (and other details) will be proven independently in the more general case of $\Z_2^n$-morphisms of $\Z_2^n$-supermanifolds.\medskip

In the sequel we write $\phi^*(x'^ i)=\zf^i(x)+j^i(x,\xi)$, with $j^i(x,\zx)=\sum_{\zs(\mu)=0} f^i_\mu(x) \zx^\mu\in J$. For any $$g(x',\xi')=\sum_{|\zn|\ge 0}g_\zn(x')\xi'^\zn\in \Ci(V')[[\zx'^1,\dots,\zx'^{q'}]]\;,$$ we set \be\label{Pullback}(\phi^*(g))(x,\xi)=\sum_{|\zn|\ge 0} \phi^*(g_\zn(x')) (\zvf^*(\xi'))^\zn\;,\ee where
\be\label{FormTayl} \phi^*(g_\zn(x'))=g_\zn(\phi^*(x'))=g_\zn(\zf(x)+j(x,\zx))= \sum_{|\za|\ge 0} \frac{1}{\za!}\; (\p ^{\za}_{x'}g_{\zn})(\zf(x))\; j^{\za}(x,\xi)\;\ee is a formal Taylor expansion; we use here the multiindex notation: $j^\za=(j^1)^{\za^1}\ldots (j^{p'})^{\za^{p'}}\in J^{|\za|}$. In fact the {\small RHS} of (\ref{FormTayl}) is a series of series and it could lead to rearranged series with non-converging series of $\Ci(V)$-coefficients. However, any type of monomial in the formal variables $\xi^a$ contains a fixed number $N$ of parameters. As the terms indexed by $|\za|>N$ contain at least $N+1$ parameters, they not contribute to the considered monomial. The coefficient of the latter is therefore a finite sum in $\Ci(V)$, so that the {\small RHS} of (\ref{FormTayl}) is actually a series in $\Ci(V)[[\zx^1,\dots,\zx^{q}]]$. The same argument can be used for the {\small RHS} of (\ref{Pullback}).\medskip

It is quite easily seen that the thus defined pullback map $\phi^*$ is a unital (obvious) graded $\R$-algebra morphism. As for the degree of $\phi^*$, note that $j^i$ is of degree $0$, so that $\zvf^*(g_\zn(x'))$ has $\Z_2^n$-degree $0$; Equation (\ref{Pullback}) allows now to see that $\zvf^*$ is of degree $0$. To prove that $\zvf^*$ is an algebra morphism, we first show that its restriction (\ref{FormTayl}) respects multiplication. If $g_\zn,h_\zr\in\Ci(V')$, we get 
\beas
\phi^*(g_\zn h_\zr)
&=&\sum_{\za} \frac{1}{\za!}\; \p ^{\za}_{x'}(g_{\zn}h_\zr)\; j^{\za} \\
&=&\sum_\za\sum_{\zb+\zg=\za} \frac{1}{\za!}\frac{\za!}{\zb!\zg!}\; \p^{\zb}_{x'}g_{\zn}\,\p^\zg_{x'}h_\zr\; j^{\zb+\zg}\\
&=&\sum_\zb\sum_\zg \frac{1}{\zb!\zg!}\; \p^{\zb}_{x'}g_{\zn}\,\p^\zg_{x'}h_\zr\; j^{\zb}j^\zg\\
&=&\zvf^*(g_\zn)\;\zvf^*(h_\zr)\;,
\eeas 
where we omitted for simplicity the evaluation at $\zf(x)$, as well as the variables of $j$ (remember that $j^i$ is of degree $0$). Let now $g=\sum_\zn g_\zn \xi'^\zn$ and $h=\sum_\zr h_\zr \xi'^\zr$ be two arbitrary $\Z_2^n$-functions: 
$$gh=\sum_{\za}\sum_{\zn+\zr=\za}\pm g_\zn h_\zr\xi'^\za\;,$$ 
where the sign is due to commutation of components of $\xi'$. Thus 
\beas
\phi^*(gh)&=&\sum_{\za} \phi^*\left(\sum_{\zn+\zr=\za}\pm g_\zn h_\zr\right) (\zvf^*(\xi'))^\za\\
&=&\sum_{\za} \sum_{\zn+\zr=\za}\zvf^*(g_\zn)\; \zvf^*(h_\zr)\; (\zvf^*(\xi'))^\zn\;(\zvf^*(\xi'))^\zr\;,
\eeas 
where the sign disappears as $\zvf^*$ is of degree $0$. The conclusion follows. \medskip

As mentioned above, the precise definition of a $\Z_2^n$-morphism as a morphism of locally $\Z_2^n$-ringed spaces will be given in Section $\ref{Z2nSuperMan}$, and the preceding explanation will be completed and generalized.

\section{$\Z_2^n$-Supermanifolds}\label{Z2nSuperMan}

\subsection{Definitions}\label{Definition}

A $\Z_2^n$-supermanifold is a locally $\Z_2^n$-ringed space ({\small LZRS}) that is locally modelled on a $\Z_2^n$-superdomain. For details on the category $\tt LZRS$ of {\small LZRS}, we encourage the reader to have a look at Section \ref{App2}. In the following, the elements of $\Z_2^n$, $n\in\N$, are considered as ordered {lexicographically}.

\begin{definition}[Ringed space definition]\label{DefZSupMan} A (smooth) \emph{$\Z_2^n$-supermanifold} $\cM$ of dimension $p|\mathbf{q}$, $p\in\N$, $\mathbf{q}=(q_1,\ldots,q_{2^n-1})\in \N^{2^n-1}$, is a {\small LZRS} $(M,\cA_M)$ that is locally isomorphic to a $\Z_2^n$-superdomain $\Ci_{\R^p}[[\xi^1,\ldots,\xi^q]]$, where $q=|\mathbf{q}|$, where $\xi^1,\ldots,\xi^q$ are formal variables of which $q_k$ have the $k$-th degree in $\Z_2^n\setminus\{0\}$, and where $\Ci_{\R^p}$ is the function sheaf of the Euclidean space $\R^p$. \end{definition}

The mentioned isomorphisms are of course invertible morphisms in $\tt LZRS$.\medskip

Many geometric concepts can be glued from local pieces: they can be defined via a cover by coordinate systems, with specific coordinate transformations that satisfy the usual cocycle condition. The same holds for $\Z_2^n$-supermanifolds. Roughly, a $\Z_2^n$-supermanifold of dimension $p|\mathbf{q}$ can be viewed as a second-countable Hausdorff topological space $M$ surrounded by `a cloud of formal stuff', which is locally (with respect to the topology of $M$) described by coordinate systems $(x,\xi)$, where $x=(x^1,\ldots,x^p)\in U\subset \R^p$ is of degree $0$ (and can be viewed as a homeomorphism $x(m)\rightleftarrows m(x)$ between $U$ an open subset of $M$ -- which is often also denoted by $U$) and $\xi=(\xi^1,\ldots,\xi^q)$ are formal variables as in Definition \ref{DefZSupMan}; further, the coordinate transformations respect the $\Z_2^n$-degree and satisfy the cocycle condition.\medskip

The rigorous alternative definition of $\Z_2^n$-supermanifolds follows naturally from this idea. It is similar to the atlas description of a supermanifold \cite{DAL}.

\begin{definition} A \emph{chart} (or coordinate system) over a (second-countable Hausdorff) topological space $M$ is a {\small LZRS} $${\cal U}=(U,\Ci_{U}[[\xi^1,\ldots,\xi^q]]),\; U\subset\R^p, p,q\in\N\;,$$ together with a homeomorphism $\psi:U\to \psi(U)$, where $\psi(U)$ is an open subset of $M$.\end{definition}

Given two charts $({\cal U}_\za,\psi_\za)$ and $({\cal U}_\zb,\psi_\zb)$ over $M$, we will denote by $\psi_{\zb\za}$ the homeomorphism $$\psi_{\zb\za}:=\psi_\zb^{-1}\psi_\za:V_{\zb\za}:=\psi_\za^{-1}(\psi_\za(U_\za)\cap \psi_\zb(U_\zb))\to V_{\za\zb}:=\psi_\zb^{-1}(\psi_\za(U_\za)\cap \psi_\zb(U_\zb))\;.$$

Whereas in classical Differential Geometry the coordinate transformations are completely defined by the coordinate systems, in ($\Z_2^n$-)Supergeometry, they have to be specified separately.\medskip

\begin{definition} A \emph{coordinate transformation} between two charts $({\cal U}_\za,\psi_\za)$ and $({\cal U}_\zb,\psi_\zb)$ over $M$ is an {isomorphism} of {\small LZRS} $\Psi_{\zb\za}=(\psi_{\zb\za},\psi^*_{\zb\za}):{\cal U}_{\za}|_{V_{\zb\za}}\to {\cal U}_{\zb}|_{V_{\za\zb}},$ where the source and target are restrictions of `sheaves' (note that the underlying homeomorphism is $\psi_{\zb\za}$).

An \emph{atlas} over $M$ is a covering $({\cal U}_\za,\psi_\za)_\za$ by charts together with a coordinate transformation for each pair of charts, such that the usual {cocycle} condition $\Psi_{\zb\zg}\Psi_{\zg\za}=\Psi_{\zb\za}$ holds (appropriate restrictions are understood).\end{definition}

\begin{definition}[Atlas definition] A (smooth) \emph{$\Z_2^n$-supermanifold} $\cM$ is a second-countable Hausdorff topological space $M$ together with a preferred atlas $({\cal U}_\za,\psi_\za)_\za$ over it. \end{definition}

\subsection{Rationale}\label{Rationale}

Let us consider the case $n=2\,$, $p|q_1|q_2|q_3=1|1|1|1$, and assume for simplicity that the underlying topological space $M$ carries a smooth manifold structure (we prove later that the underlying topological space of any $\Z_2^n$-supermanifold carries a smooth structure). We use notation from Examples \ref{FundaExa} and \ref{FundaExa2}; in particular $(x,\xi,\zh,\zy)$ are of degree $((0,0),(0,1),(1,0),(1,1))$. A coordinate transformation $(x,\xi,\zh,\zy)\rightleftarrows (y,\za,\zb,\zg)$ is then of the form \be\label{CoordTransfo}(a)\left \{\begin{array}{c} y=\sum_r f^y_r(x)\zy^{2r}+\sum_r g^y_r(x)\zy^{2r+1}\xi\zh\\ \za=\sum_r f^\za_r(x)\zy^{2r}\xi+\sum_r g^\za_r(x)\zy^{2r+1}\zh \\ \zb=\sum_r f^\zb_r(x)\zy^{2r}\zh+\sum_r g^\zb_r(x)\zy^{2r+1}\xi \\ \zg=\sum_r f^\zg_r(x)\zy^{2r+1}+\sum_r g^\zg_r(x)\zy^{2r}\xi\zh\end{array}\right.\;(b)\left\{\begin{array}{c} x=\sum_r F_r^x(y)\zg^{2r}+\sum_r G_r^x(y)\zg^{2r+1}\za\zb\\ \xi=\sum_r F_r^\xi(y)\zg^{2r}\za+\sum_r G_r^\xi(y)\zg^{2r+1}\zb\\ \zh=\sum_r F_r^\zh(y)\zg^{2r}\zb+\sum_r G_r^\zh(y)\zg^{2r+1}\za\\ \zy=\sum_r F_r^\zy(y)\zg^{2r+1}+\sum_r G_r^\zy(y)\zg^{2r}\za\zb\end{array}\right.\;\ee where the functions in $x$ and $y$ are smooth.

The substitution of (\ref{CoordTransfo})(a) in a local function 
\be\label{LocFun}f(y,\za,\zb,\zg)=\sum_{ i,j\in\{0,1\},\;r\in\N} f_{ijr}(y)\za^i\zb^j\zg^r\;\ee 
leads to a function $g(x,\xi,\zh,\zy)$ in the initial variables -- the pullback of $f$. As mentioned before, to transform 
$$\label{Tay1}f_{ijr}\left(\sum_r f_r^y(x)\zy^{2r}+\sum_r g_r^y(x)\zy^{2r+1}\xi\zh\right)\;,$$
we detach the independent term $f_0^y(x)$ from the series $j(x,\xi,\zh,\zy)$ of all the remaining terms and set 
$$\label{Tay2}f_{ijr}\left(f_0^y(x)+j(x,\xi,\zh,\zy)\right)=\sum_{n}\frac{1}{n!}\;\,\frac{\op{d}^n f_{ijr}}{\op{d}y^n}(f_0^y(x))\;j^n(x,\xi,\zh,\zy)\;\,$$
(remember that according to our simplifying assumptions $f_{ijr}$ is a function of a unique variable $y$).

It is now quite obvious that a coordinate transformation (\ref{CoordTransfo}) in a $\Z_2^n$-supermanifold induces a coordinate transformation $y=f_0^y(x)$, $x=F_0^x(y)$ in the base manifold. Indeed, since the transformations (\ref{CoordTransfo}) are inverse, we get $x$ when substituting (\ref{CoordTransfo})(a) in 
$$x=\sum_r F_r^x(y)\zg^{2r}+\sum_r G_r^x(y)\zg^{2r+1}\za\zb\;,$$ 
i.e. all the terms of the {\small RHS} that contain, after substitution, at least one parameter cancel, whereas the unique parameter independent term $F_0^x(f_0^y(x))$ coincides with $x$. Similarly, $f_0^y(F_0^x(y))=y\,.$\medskip

From the `atlas standpoint', a global $\Z_2^2$-superfunction $f\in\cA_M^{(i,j)}(M)$ of degree $(i,j)\in\Z_2^2$, is a family $f(y,\za,\zb,\zg)$, over all coordinate systems $(y,\za,\zb,\zg)$, of local functions of degree $(i,j)$, such that, when substituting (\ref{CoordTransfo})(a) in $f(y,\za,\zb,\zg)$, we get the function $f(x,\xi,\zh,\zy)$ associated to the coordinate system $(x,\xi,\zh,\zy)$ -- just as a global function $g\in\Ci_M(M)$ is a family $g(y)$, over the induced coordinate systems $(y)$, such that when substituting $y=f_0^y(x)$ in $g(y)$, we get $g(x)$. The degree $(i,j)$ is compatible with the coordinate transformations as the latter respect the degrees.\medskip

For any global $\Z_2^2$-superfunction in $\cA_M(M)$, i.e. any family of `glueable' local functions $f(y,\za,\zb,\zg)$, see (\ref{LocFun}), the induced family $f_{000}(y)$ defines a global base function in $\Ci_M(M)$. Indeed, in view of what has been said, it is easily checked that the gluing property of the family $f(y,\za,\zb,\zg)$ entails that $f_{000}(f_0^y(x))=f_{000}(x)\,.$

\begin{remark}\label{RemRat} This means that the canonical projections of the local expressions of a global function glue to give a global base function. In particular projection commutes with restriction. Therefore, projection is an algebra morphism.\end{remark}

\section{Examples}

 \begin{example} For $n=1$, we recover classical supermanifolds. Indeed, in this case there are no formal variables that bear powers higher than 1 and formal series are thus just polynomials.
\end{example}

\begin{example} We already mentioned that the tangent bundle $\sss T \cM$ to a $\Z_2$-supermanifold $\cM=(M,\cA_M)$ gives rise to a $\Z_2^2$-supermanifold. Indeed, the parities of local coordinates $(x^i,\zx^a)$ on $\cM$ induce canonically parities of the adapted system of coordinates $(x^i,\zx^a,\dot x^j,\dot \zx^b)$ on ${\sss T \cM}$ in which
$(x^i,\dot x^j)$ are even and $(\zx^a,\dot\zx^b)$ are odd. But $\sss T \cM$ is also a vector bundle what induces an additional $\N$-grading in which $(\dot x^j,\dot \zx^b)$ are of degree 1.
Using the canonical monoid homomorphism from $\N$ to $\Z_2$, we get a $\Z_2^2$-grading in which
$(x^i,\zx^a,\dot x^j,\dot \zx^b)$ have the bidegrees $((0,0),(0,1),(1,0),(1,1))$.
We can find an atlas whose coordinate changes respect the bidegrees; hence, we deal with a $\Z_2^2$-supermanifold. As the changes of coordinates are linear in $(\dot x^j,\dot \zx^b),$ the algebra of $\Z_2^2$-superfunctions which are polynomial in the latter variables is well-defined.
It can be identified with the algebra $\zW_D({\cM})$ of Deligne super differential forms on $\cM$.
Since it is dense in the whole algebra of $\Z_2^2$-superfunctions on $\sss T\cM$, the latter can be identified with the corresponding completion
$$\widehat{\zW}_D(\cM)=\prod_{k\ge 0}\w^k\zW^1\,.$$

\begin{prop} The tangent bundle $\sss T\cM$ of a $\Z_2$-supermanifold $\cM=(M,\cA_M)$, interpreted as ringed space $(M,\widehat{\zW}_D(\cM))$, is a $\Z_2^2$-supermanifold.
\end{prop}

\end{example}

\begin{example} Let now
$$
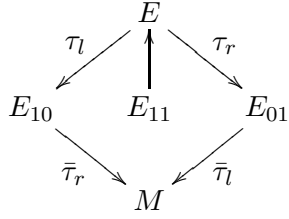
\label{algebroid}
        {\xymatrix@R-1mm @C-1mm{ & E \ar[ld]_*{\zt_l} \ar[rd]^*{\zt_r} & \cr
        E_{10} \ar[rd]_*{\bar\zt_r} & E_{11} \ar[u] & E_{01} \ar[ld]^*{\bar\zt_l}  \cr & M  & }}
$$
be a double vector bundle with the side bundles $E_{01},E_{10}$ and the core bundle $E_{11}$.
This corresponds to a choice of two commuting Euler vector fields $\nabla_1,\nabla_2$ on $E$ \cite{GR09}. We can choose an atlas with bihomogeneous local coordinates, say $(x,\zx,\zh,\zy)$ of bidegrees $(0,0),(0,1),(1,0),(1,1)$, respectively. Moreover, all coordinate changes have the form
\beas
x'&=&\zf(x)\,,\\
\zx'&=&a(x)\zx\,,\\
\zh'&=&b(x)\zh\,,\\
\zy'&=&c(x)\zy+d(x)\zx\zh\,,
\eeas
and thus respect the bigrading. We can now `superize' assuming that these coordinates satisfy the sign rules of the `scalar product' in $\Z_2^2$. As the coordinate changes respect the bidegrees, this is consistent and leads to a $\Z_2^2$-supermanifold $\zP E$.
%Actually, we can construct consistently a $\Z_2^2$-supermanifold $\zP E[\ze]$ using any factor $\ze$ on $\Z_2^2$. This construction is completely canonical. The only delicate point with respect to standard superization of a vector bundle is the appearance of the product $\zx\zh$.
In the super case, we have to fix the ordering, as its change may result in changing the sign (see discussion in \cite{GR09}).\medskip

All this can be generalized to $n$-tuple vector bundles if we fix a lexicographic ordering in $\Z_2^n$ relative to an ordering of the corresponding Euler vector fields.

\begin{prop} The superization of an $n$-vector bundle, $n\ge 1$, is a $\Z_2^n$-supermanifold.\end{prop}

Note that certain superizations of $n$-tuple vector bundles have been considered also by Voronov.
\end{example}

\section{Morphisms of $\Z_2^n$-supermanifolds}

\subsection{Embedding of the smooth base manifold}

We already mentioned that global $\Z_2^n$-functions project consistently to the base, see Remark \ref{RemRat}. In the present section, we make this observation more precise.

\begin{prop} The base topological space $M$ of any $\Z_2^n$-supermanifold $\cM=(M,\cA_M)$ of dimension $p|\mathbf{q}$ carries a smooth manifold structure of dimension $p$, and there exists a short exact sequence of sheaves $$0\to {\cal J}_M\to \cA_M\stackrel{\ze}{\to }\Ci_M\to 0\;.$$\end{prop}

\begin{proof} Let $V\subset\R^{p}$ be open, let $f \in \cO(V)=\Ci(V)[[\xi^1,\ldots,\xi^q]]$, let $x\in V$ and $k\in\R$. In view of Corollary \ref{p1}, the $\Z_2^n$-function $f-k$ is not invertible, in any neighborhood of $x$ in $V$, if and only if its independent term $f_0-k$ is not invertible, in any neighborhood of $x$, i.e. if and only if $k=f_0(x)$. Hence, for any $V\subset\R^p$, any $f\in\cO(V)$ and any $x\in V$, there exists a unique $k\in \R$, such that $f-k$ is not invertible, in any neighborhood of $x$ in $V$. Since $\cA_M$ is locally isomorphic to $\cO_{\R^p}$, the same property holds in $\cA_M$. For any open $U\subset M$, for any $f\in\cA(U)$ and any $m\in U$, the unique $k\in\R$ such that $f-k$ is not invertible, in any neighborhood of $m$, is denoted by $\ze_U(f)(m)$. If $m$ runs through $U$, we obtain a function $\ze_U(f):U\to \R$, and if $f$ runs through $\cA(U)$, we get a map
$\eps_{U} : \cA(U)\to \cF (U)$, where ${\cal F}(U)=\op{im}\ze_U$ is the algebra of these functions. Actually $\ze_U$ is a surjective algebra morphism and the short sequence of algebras
$$ 0 \raa  \cJ(U)  \raa  \cA(U)\raa {\cF}(U) \raa   0\;,$$  where $\cJ(U)=\ker \eps_{U},$ is exact. In fact $\cJ_M:U\to \cJ(U)$ is a subsheaf of $\cA_M$. On the other hand, it is clear that the presheaf ${\cal F}_M$ is locally isomorphic to $\Ci_{\R^p}$ and is thus locally a sheaf. Hence, ${\cal F}_M$ generates a sheaf $\frak{F}_M$ which is locally isomorphic to $\Ci_{\R^p}$ and thus implements a $p$-dimensional differential manifold structure on $M$ such that $\Ci_M\simeq \frak{F}_M$. Since the sequence of sheaves $$ 0 \raa \cJ_M \raa \cA_M \stackrel{\ze}{\raa} \Ci_M \raa  0 \;$$ is exact, we have $\cA_M/{\cal J}_M\simeq \Ci_M$. For details on sheaves, we refer the interested reader to Section \ref{App3}. See also \cite{Var}.\end{proof}

% To get $\frak F$ we apply the gluing of sheaves, see \ref{App3}. The conditions are satisfied with $\zf_{ji}=\id$, and in our case a section in $\ze_U(f)\in \cF(U)$ defines a section $(\ze_U(f)|_{U\cap U_i})_i\in {\frak F}(U)$, which can be identified with $\ze_U(f)$. Hence, the map $\ze_U$ valued in $\cF(U)$ and the map $\ze_U$ valued in ${\frak F}(U)$ are the same. Also there is no problem as concerns surjectivity in the case of $\frak F$, since we deal then with an exact sequence of sheaves.

\subsection{Continuity of morphisms}\label{ContMorph}

In the `ringed space definition' of $\Z_2^n$-supermanifolds the requirement that $(M,\cA_M)$ be local is actually redundant -- in view of the local model. The unique maximal homogeneous ideal of $\cA_m$, $m\in M$, will be denoted by $\frak{m}_m$.

The key-fact about morphisms of $\Z_2^n$-supermanifolds is a generalization of Section \ref{MorphTheo}, see below. This result can be proved due to the continuity of morphisms with respect to the topologies induced by the ideals ${\cal J}(U)\subset\cA(U)$ and $\frak{m}_m\subset \cA_m$.

In this section, we prove these continuities.

\newcommand{\cB}{{\cal B}}

\begin{definition} A \emph{morphism of $\Z_2^n$-supermanifolds} or \emph{$\Z_2^n$-morphism} is a morphism of the underlying locally $\Z_2^n$-ringed spaces.\end{definition}

This means that the category ${\tt ZSMan}$ of $\Z_2^n$-supermanifolds is a full subcategory of the category ${\tt LZRS}$, see Section \ref{App2}.\medskip

We first show that $\Z_2^n$-morphisms commute with the projections $\ze$ onto the bases:

\begin{prop} \label{commepspsi}
Let $$ \Psi=(\psi,\psi^*) : {\cal M} =(M,\cA_M)  \to  {\cal N} =(N,\cB_N) $$ be a morphism of $\Z_2^n$-supermanifolds, let $V\subset N$ be an open subset, and $U=\psi^{-1}(V)$. Then,
$$ \ze_{U} \circ \psi^{*}_V = \psi^{*}_V \circ \ze_{V}\;, $$
 where the {\small LHS} pullback of $\Z_2^n$-functions is given by the second component of $\Psi$ and where the {\small RHS} pullback of classical functions is equal to $- \circ \psi|_U $ and thus given by the first component of $\Psi$.\end{prop}

\begin{proof} Let $t\in \cB(V)$ and $m \in U$. If we set $s=\psi^{*}_V(t) \in \cA(U)$, we have to show that
$$ \ze_{U} (s) (m) = \ze_{V}(t)(\psi(m))\;. $$
The {\small RHS} of this equation is, by definition, the unique $k\in\R$ such that $t-k$ is not invertible, in any neighborhood of $\psi(m)$. It suffices thus to prove that the {\small LHS} has this property. Suppose that $t- \ze_{U} (s) (m)$ is invertible in some neighborhood of $\psi(m)$. Then, since $\psi^{*}_V$ is a unital $\R$-algebra morphism, $$\psi^{*}_V \lp t- \ze_{U} (s) (m) \rp=\psi^{*}_V (t) - \ze_{U} (s) (m)\; \psi^{*}_V(1) \\ = s - \ze_{U} (s) (m)$$ is invertible in some neighborhood of $m$ -- a contradiction.\end{proof}

\begin{cor} For any $\Z_2^n$-supermanifold ${\cal M} =(M,\cA_M)$ and any point $m\in M$, the unique maximal homogeneous ideal ${\frak m}_m$ of $\cA_m$ is given by \be\label{MaxIdea}{\frak m}_m=\{[f]_m:(\ze f)(m)=0\}\;.\ee\end{cor}

\begin{proof} If $$\Phi=(\zvf,\zvf^*):(U,\cA_M|_U)\stackrel{\sim}{\to}(V,\Ci_{\R^p}|_V[[\xi^1,\ldots,\xi^q]])$$ denotes an isomorphism in $\tt LZRS$, with $m\in U$, we have ${\frak m}_m=\zvf^*_m{\frak m}_x$, where $x=\zvf(m)$. It now suffices to apply Proposition \ref{commepspsi}.\end{proof}

As mentioned above, we need not assume that $(M,\cA_M)$ is local. Then $\Phi$ is only an isomorphism in $\tt ZRS$ and we cannot ask that $\zvf^*_m$ respects maximal ideals. However, since $\zvf^*_m$ is an isomorphism of graded unital $\R$-algebras, $\zvf^*_m{\frak m}_x$ \emph{is} the unique maximal homogeneous ideal of $\cA_m$.

The next result is the announced $\cJ$- and $\frak m$-continuity theorem for $\Z_2^n$-morphisms. It shows in particular that $\Z_2^n$-morphisms automatically respect maximal ideals, so that this requirement is actually redundant in the definition of $\Z_2^n$-morphisms.

\begin{cor}\label{Claim1} Any $\Z_2^n$-morphism $\Psi=(\psi,\psi^*):\cM=(M,\cA_M)\to \cN=(N,\cB_N)$ is continuous with respect to $\cal J$ and $\frak m$, i.e., for any open $V\subset N$ and any $m\in M$, we have
$$ \psi^{*}_V\lp \cJ_N(V) \rp \subset \cJ_M({\psi^{-1}(V)})\;\;\text{and}\;\; \psi^{*}_{m}\lp \frak{m}_{\psi(m)}\rp \subset \frak{m}_{m}\;.\label{stalkcond}$$
\end{cor}

\begin{proof} This result is a direct consequence of the definition $\cJ=\ker\ze$, Equation (\ref{MaxIdea}), and Proposition \ref{commepspsi}.\end{proof}

\begin{cor} The base map $\psi:M\to N$ of any $\Z_2^n$-morphism $\Psi:(M,\cA_M)\to (N,\cB_N)$ is smooth.\end{cor}

\begin{proof} Let $m \in M$, let $(V,y=(y^1,\ldots,y^u))$ be a classical chart of $N$ around $\psi(m)$, and set $U=\psi^{-1}(V)$. For any $g\in \Ci(V)$, there is $t\in {\cal B}(V)$ (just restrict $V$), such that
$$ g \circ \psi=  \psi^{*} \lp \ze_{V}(t) \rp = \ze_{U} \lp \psi^{*}(t) \rp\in \Ci(U)\;.$$
In particular, for $g=y^{j}$, we get $\psi^{j}= y^{j} \circ \psi \,\in \Ci(U)\;,$ so that $\psi\in \Ci(U,N)$ and, since $U$ is a neighborhood of an arbitrary point $m\in M$, $\psi \in \Ci(M,N)$.  \end{proof}

\subsection{Completeness of the $\Z_2^n$-function sheaf and the $\Z_2^n$-function algebras}

In standard Supergeometry, the decreasing filtration $\cA\supset \cJ\supset \cJ^2\supset\ldots$ of the structure sheaf $\cA$ of a supermanifold $\cM=(M,\cA)$ by the sheafs of ideals $\cJ^k$ \cite{Var}, induces embeddings $$M\hookrightarrow \cM^2\hookrightarrow \cM^3\hookrightarrow\ldots\hookrightarrow\cM\;,$$ where $\cM^k=(M,\cA/\cJ^k)$ is the superspace characterized by the sheaf $\cA/\cJ^k$ \cite{Man} (see also Section \ref{App3}). % Varadarajan mentions that the $J^k$ are sheaves, Manin that the $\cO/\cJ^k$ are sheaves. In my current understanding, $J^k$ is a $B$-sheaf for the basis of chart domains $U_\za$. Indeed, $J^k(U_\za)$ is the series with $k$ parameters at least. If we consider $s_\za\in J^k(U_\za)$, we can choose coordinates $(z,\theta)$ in $U_{\za\zb}$, change coordinates in $s_\za$ and $s_\zb$ and thus get, as the number of parameters cannot decrease, two elements in $J^k(U_\za)$ and $J^k(U_\zb)$ written in the coordinates $(z,\theta)$. As the latter coincide in $U_{\za\zb}$, their coefficients coincide there and can be glued. Hence an element in $J^k(U_\za\cup U_\zb)$ which extends the original elements. What we consider here is the extension of the $B$-sheaf $J^k$ to a sheaf. The point is that `big' sections are not only families of `small' sections that coincide on intersections, but that we deal with an extension, i.e. on $U_\za$ we recover what we started from.

Let now $\cM=(M,\cA)$ be a $\Z_2^n$-supermanifold. The decreasing filtration $\cA\supset \cJ\supset \cJ^2\supset\ldots$ gives rise to an inverse system $$\cA/\cJ\leftarrow \cA/\cJ^2\leftarrow \cA/\cJ^3\leftarrow\ldots $$ of sheaves of algebras (we prove at the end of this subsection that the quotient presheaves $\cA/\cJ^k$ are actually sheaves). Since a limit is a universal cone, there is a sheaf morphism $\varprojlim_k\cA/\cJ^k\leftarrow \cA$. Moreover, as a limit in a category of sheaves is just the corresponding limit in the category of presheaves (which is computed objectwise), we get, for any $\Z_2^n$-superchart domain $U_\za$, $$\lp \varprojlim_k\cA/\cJ^k \rp (U_\za)=\varprojlim_k\cA(U_\za)/\cJ^k(U_\za)\simeq \cA(U_\za)\;,$$ see Equation (\ref{Completeness}). It follows that $$\varprojlim_k\cA/\cJ^k\simeq\cA\;$$ in the category of sheaves, so that the structure sheaf $\cA$ is complete with respect to the filtration implemented by $\cJ.$ % Note that $\cJ^k$ is the extension of a $B$-sheaf.
Since isomorphic sheaves have isomorphic sections, we thus obtain, for any $U\subset M$, $$\varprojlim_k\cA(U)/\cJ^k(U)\simeq\cA(U)\;,$$ i.e. all function algebras $\cA(U)$, $U\subset M,$ are Hausdorff-complete with respect to the filtration induced by $\cJ(U)$, see Section \ref{App1}.

\begin{prop} The $\Z_2^n$-function sheaf $\cA_M$ (resp., the $\Z_2^n$-function algebra $\cA_M(U),$ $U\subset M$) of a $\Z_2^n$-supermanifold $(M,\cA_M)$ is Hausdorff-complete with respect to the $\cJ_M$-adic (resp., $\cJ_M(U)$-adic) topology.\end{prop}

\begin{proof} It suffices to show that the presheaves $\cA/\cJ^k$, $k\ge 1$, are in fact sheaves.

Let $s,t\in\cA(U)/\cJ^k(U)$, $U\subset M$, coincide on an open cover $(V_\za)_\za$ of $U$. Then $s|_{V_\za}=t|_{V_\za}+j^k_{V_\za}$, with $j^k_{V_\za}\in \cJ^k(V_\za)$. The latter thus coincide on the intersections $V_\za\cap V_\zb$. Their gluing provides $j^k\in \cJ^k(U)$ and $s|_{V_\za}=t|_{V_\za}+j^k|_{V_\za}$. Since $\cA$ is a separated presheaf (even a sheaf), we get $s=t+j^k$ in $\cA(U)$ and $s=t$ in $\cA(U)/\cJ^k(U)$.

As for the second sheaf property, consider a family $s_\za\in\cA(V_\za)/\cJ^k(V_\za)$ such that $s_\za=s_\zb$ on $V_\za\cap V_\zb$. Let $\zp_\za$ be a partition of unity of $\cM$ that is subordinated to $V_\za$ (see Section \ref{PartUnitQuotShea}). Any product $\zp_\za s_\za$ of $s_\za$ by $\zp_\za\in\cA^{0}(U)$ ($\op{supp}\zp_\za\subset V_\za$) is well-defined in $\cA(U)/\cJ^k(U)$ ($\op{supp}(\zp_\za s_\za)\subset V_\za$). Then $s=\sum_\za\zp_\za s_\za\in\cA(U)/\cJ^k(U)$ and $s|_{V_\za}=s_\za$.         \end{proof}

\subsection{Fundamental theorem of $\Z_2^n$-morphisms}

In this section we prove an extension of the main statement of Section \ref{MorphTheo}.

\subsubsection{Statement}

Let  $$ \Psi =(\psi, \psi^{*}) :  {\cal M} =(M,\cA_M) \raa {\cal V}^{\,u|\mathbf{v}} =(V, \Ci_V[[\zh^1,\ldots,\zh^v]]) $$ be a $\Z_2 ^n$-morphism valued in a $\Z_2^n$-superdomain. Denote by $y=(y^1,\ldots,y^u)$ the coordinates of $V$ and by $\zt_b\in \Z_2^n\setminus\{0\}$ the degree of $\zh^b$. Then the functions $$s^{j}:= \psi^{*}_Vy^{j}\;\text{and}\;\zeta^{b}:= \psi^{*}_V\eta^{b}$$ (for all $j$ and $b$)
satisfy
\be\label{C1} s^{j} \in \cA^{0}_M(M)\;,\zeta^{b} \in \cA_M^{\zt_b}(M)\;\text{and}\;\lp  \ze s^{1}, \ldots, \ze s^{u}\rp(M)\subset V\;.\ee

Actually, the pullbacks $(s^j,\zeta^b)$ (all $j$ and $b$) of the coordinate functions $(y^j,\zh^b)$ completely determine the $\Z_2^n$-morphism. More precisely:

\begin{thm}[Fundamental theorem of $\Z_2^n$-morphisms]\label{FundaTheoSuperm} If ${\cal M}=(M,\cA_M)$ is a $\Z_2^n$-supermani-fold, ${\cal V}^{\,u|\mathbf{v}}$ a $\Z_2^n$-superdomain as above, and if $(s^j,\zeta^b)$ is an $(u+v)$-tuple of $\Z_2^n$-functions in $\cA_M(M)$ that satisfy the conditions (\ref{C1}), there exists a unique morphism of $\Z_2^n$-supermanifolds
$\Psi =(\psi, \psi^{*}) :  {\cal M} \raa {\cal V}^{u|\mathbf{v}},$ such that $ s^{j}= \psi^{*}_Vy^{j}$ and $\zeta^{b}= \psi^{*}_V\eta^{b}$.\end{thm}

The proof requires some preparatory work.

\subsubsection{Polynomial approximation of $\Z_2^n$-functions}

We first describe the ${\frak m}_m$-adic topology of $\cA_m$, $m\in M$.\medskip

When taking an interest in the stalks $\cA_m$ of the function sheaf of a $\Z_2^n$-supermanifold $(M,\cA_M)$ of dimension $p|\mathbf{q}$, we can choose a centered chart $(x,\xi)$ around $m$ and work in a $\Z_2^n$-superdomain ${\cal U}^{p|\mathbf{q}}$ associated with a convex open subset $U\subset \R^p,$ in which $m\simeq x=0$. Since ${\frak m}_m=\{[f]_m:\ze(f)(m)=0\}$, a Taylor expansion (with remainder) around $m\simeq x=0$ of the coordinate form of $\ze(f)$ shows that $${\frak m}_m\simeq \{[f]_0: f(x,\xi) = 0(x) + \sum_{|\zm|>0}f_\zm(x)\xi^\zm \}\;,$$ where $0(x)$ are terms of degree at least 1 in $x$.

\begin{lem}\label{Claim4} For any $m\in M$, the basis of neighborhoods of $\,0$ in the ${\frak m}_m$-adic topology of $\cA_m$ is given by 
$$
{\frak m}_m^{k+1}=\{[f]_0: f(x,\xi) = \sum_{0\le|\zm|\le k}0_\zm(x^{k-|\zm|+1})\xi^\zm+\sum_{|\zm|>k}f_\zm(x)\xi^\zm\}\quad(k\ge 0)\;,
$$ 
where notation is the same as above. \end{lem}

\begin{proof} The inclusion $\subset$ is obvious. For $\supset$, note that all the terms whose number $|\zm|$ of generators is $\le k$ belong to ${\frak m}_m^{k+1}$; as for the series over $|\zm|>k$, write it in the form $$\sum \xi^{i_1}\ldots\xi^{i_k}\sum_{|\zm|>k}f_\zm(x)\,\xi^{\zm-e_{i_1}-\ldots-e_{i_k}}\;,$$ where $e_a$ is the canonical basis of $\R^q$ and where the sum is over all combinations (with repetitions) of $k$ of the $q$ elements $\xi^a$. Each term of this finite sum is in ${\frak m}_m^{k+1}$ and so is the sum itself.\end{proof}

Roughly speaking, since $\psi^*_{V}$ is an algebra morphism, the data $ \psi^{*}_Vy^{j}=s^{j}$ and $ \psi^{*}_V\eta^{b}=\zeta^{b}$ uniquely determine the pullback $\psi^{*}_VP$ of any section $P\in \op{Pol}_V(V)[[\zh^1,\ldots,\zh^v]]$ with polynomial coefficients. Hence the quest for an appropriate polynomial approximation of an arbitrary section. Let us emphasize that here and in the following, the term `polynomial section' refers to a formal series $\sum_{|\zm|\ge 0} P_\zm(x)\xi^\zm$ in the parameters $\xi^a$ with coefficients $P_\zm(x)\in \op{Pol}_V(V)$ that are polynomial in the base variables $x^i$.

\begin{thm}[Polynomial approximation] \label{Claim3}
Let $m\in M$ be a base point of a $\Z_2^n$-supermanifold $(M,\cA_M)$ and let $f\in\cA(U)$ be a $\Z_2^n$-function defined in a neighborhood $U$ of $m$. For any fixed degree of approximation $k\in\N\setminus\{0\}$, there exists a polynomial $P=P(x,\xi)$ such that $$[f]_{m}-[P]_{m} \in \frak{m}_{m}^{k}\;.$$
\end{thm}

In this statement the polynomial $P$ depends on $m,$ $f$, and $k$, and the variables $(x,\xi)$ are (pullbacks of) coordinates centered at $m$.

\begin{proof}
When writing $f=f(x,\xi)=\sum_{|\zm|\ge 0}f_\zm(x)\xi^\zm$ and using a Taylor expansion of the $f_{\zm}(x)$ at $m\simeq x=0$, we get $$ f(x,\xi)= \sum_{|\zm|\ge 0} P_{\zm}(x) \xi^{\zm} + \sum_{|\zm|\ge 0} 0_{\zm}(x^{k}) \xi^{\zm}, $$ where the first sum of the {\small RHS} is the searched polynomial $P=P(x,\xi)$ and where the (germ of the) second sum belongs to ${\frak m}_m^k$.\end{proof}

\subsubsection{Proof of the Morphism Theorem}

1. Uniqueness: if the searched $\Z_2^n$-morphism $\Psi=(\psi,\psi^*)$ exists, it is necessarily unique. Indeed, let $\Psi_1=(\psi_{1},\psi_{1}^{*})$ and $\Psi_2=(\psi_{2},\psi_{2}^{*})$ be two $\Z_2^n$-morphisms defined by the same $(s^j,\zeta^b)$.

Note first that $\psi_1=\psi_2$: if we denote the coordinates of ${\cal V}^{p|\mathbf{q}}$ by $(y^j,\zh^b)$, commutation of the pullback maps with the projections to the base entails that, for all $j$, $$\psi_1^j=y^j\circ \psi_1=\ze\psi_1^*y^j=\ze s^j=\ze\psi_2^*y^j=y^j\circ\psi_2=\psi_2^j.$$

As for the comparison of $\psi_1^*$ and $\psi_2^*$, we first show (rigorously) that they coincide on polynomial sections (using continuity of the pullbacks of sections with respect to the $\cJ$-adic topology), then we show that they coincide on arbitrary section (using the polynomial approximation and continuity of the pullbacks of germs with respect to the ${\frak m}$-adic topology).

Let $W\subset V$ be open, set $\cO(W):=\Ci(W)[[\zh^1,\ldots,\zh^v]]$, $\psi:=\psi_1=\psi_2$ and $U:=\psi^{-1}(W)$, and take $$f(y,\zh)=\sum_{|\zm|\ge 0}f_\zm(y)\zh^{\,\zm}\in\cO(W)\;.$$

The sequence $$\sum_{k=0}^n\sum_{|\zm|=k}\psi_i^*\lp f_\zm(y)\rp\lp\psi_i^*\zh\rp^\zm\in\cA(U)$$ is Cauchy, see \ref{App1}. Indeed, $$\psi_i^*\sum_{|\zm|=k}f_\zm(y) \zh^{\,\zm}\in\cJ^k(U)\;,$$ in view of the $\cJ$-continuity of $\psi_i^*$. % Note that continuity was proven for the usual definition of $\cJ^k(U)$, whereas for convergence issues we have to use the sheaf-extension definition. However, an element of the usual $\cJ^k(U)$ can be identified with an element of the $\cJ^k(U)$ obtained by sheaf-extension.
Since $\cA(U)$ is Hausdorff-complete with respect to the $\cJ(U)$-topology, the considered sequence has a unique limit $$\sum_{|\zm|\ge 0}\psi_i^*\lp f_\zm(y)\rp\lp\psi_i^*\zh\rp^\zm\in\cA(U)\;.$$ It is easily seen that this limit is given by $\psi_i^*f\in\cA(U)$: the difference of the latter and the $n$-th term of the sequence equals $$\psi_i^*\sum_{|\zm|>n}f_\zm(y)\zh^{\,\zm}\in\cJ^{n+1}(U)\;.$$

If the function $f=:P$ has polynomial coefficients $f_\zm(y)=:p_\zm(y)$, we get $$\psi_i^*P=\sum_{|\zm|\ge 0}\psi_i^*\lp p_\zm(y)\rp\lp\psi_i^*\zh\rp^\zm=\sum_{|\zm|\ge 0}p_\zm(\psi_i^*y)\lp\psi_i^*\zh\rp^\zm=\sum_{|\zm|\ge 0}p_\zm(s)\zeta^{\,\zm}\;,$$ so that $\psi_1^*$ and $\psi_2^*$ coincide on polynomial functions $P\in\cO(W).$

Consider now an arbitrary function $f\in\cO(W)$, any point $m_0\in U$, as well as a $\Z_2^n$-superchart $(U_\za,\Phi_\za)$ around $m_0$. For every $m\in U_\za$, we have $\psi(m)=:n\in W$. In view of the polynomial approximation theorem \ref{Claim3}, there is, for any $k$, a polynomial $P$ such that $\psi_i^*\lp [f]_n-[P]_n\rp \in\frak{m}_m^{k+1}$. Hence, $$[\psi_1^*f-\psi_2^*f]_m=[\psi_1^*f]_m-[\psi_1^*P]_m-[\psi_2^*f]_m+[\psi_2^*P]_m\in\frak{m}_m^{k+1}\;.$$ Since $k$ is arbitrary, it follows from Lemma \ref{Claim4} that any coefficient of $\psi_1^*f-\psi_2^*f$ vanishes at $m$. As $m$ is arbitrary in $U_\za$, we see that $\lp\psi_1^*f-\psi_2^*f\rp|_{U_\za}=0$, i.e. that $\psi_1^*f-\psi_2^*f$ vanishes on an open cover of $U$: $\psi_1^*f=\psi_2^*f$ for any $f$, so $\psi_1^*=\psi_2^*$. This completes the proof of uniqueness.\medskip

2. Existence: The base map $\psi$ is defined by $\psi:=(\ze s^1,\ldots,\ze s^u)\in \Ci(M,V)$. As for the pullback $\psi^*$, let $W\subset V$ be open. To construct the graded unital $\R$-algebra morphism $$\psi^{*}_{W} : \cO(W)  \to \cA \lp \psi^{-1}(W)\rp\;,$$ we cover the open subset $\psi^{-1}(W)\subset M$ by $\Z_2^n$-superchart domains $U_\za$. In view of uniqueness, we can take $\psi^{-1}(W)=U_\za$ and build $\psi^*_{W}: \cO(W)\to \cA(U_\za)$. This construction is the same as the one described in Section \ref{MorphTheo} and similar to the proof that holds in the super-case \cite{Var}.

\section{Appendix}

\subsection{Completeness}\label{App1}

\subsubsection{Completion of an Abelian group}

Let $G$ be an Abelian group together with a decreasing filtration by Abelian subgroups $$G=G_0\supset G_1\supset G_2\supset\ldots$$ The sequence $$G/G_0=\{0\}\leftarrow G/G_1\leftarrow G/G_2\leftarrow\ldots\;,$$ where the morphism $f_{k-1,k}:G/G_k\to G/G_{k-1}$ associates to any coset $g+G_k$, $g\in G$, the coset $g+G_{k-1}\,,$ is an inverse system in the category $\tt AbGr$ of Abelian groups. The limit $$\widehat G:=\varprojlim_k G/G_k\in\tt AbGr$$ is the {\it completion} of $G$ with respect to the considered decreasing filtration. The completion $\widehat G$ is in fact a complete topological group. Indeed, the cosets $g+G_k$, $g\in G$, $k\in\N$, are a basis of a topology of $G$ called {\it topology associated to the filtration} $G_k$. This topology turns $G$ into a topological group and $\widehat G$ into a complete topological group.\medskip

If $G$ has additional structure, e.g. is a not necessarily commutative ring (resp., a module over a ring), and {if the decreasing filtration of subgroups is compatible}, i.e. $G_k\cdot G_\ell\subset G_{k+\ell}$ (resp., $G_k$ is a submodule), then the completion is itself a ring (resp., a module). If the filtration of a ring $G$ is implemented by an ideal $I$, i.e. if $G_k=I^k:=I\cdots I$, the associated topology is also referred to as the {\it $I$-adic topology}.

\subsubsection{Hausdorff-completeness}

Let $\zk$ be a commutative unital ring and let $$M=M_0\supset M_1\supset M_2\supset\ldots$$ be a decreasing filtration of a $\zk$-module $M$ by $\zk$-submodules $M_k$. The sequence of $\zk$-modules $$0\to \cap_k M_k\stackrel{i}\longrightarrow M\stackrel{p}{\longrightarrow} \varprojlim_kM/M_k\;$$ is exact. Indeed, the map $p$ associates to $m\in M$ the sequence $(m+M_1,m+M_2,\ldots)\in\varprojlim_kM/M_k;$ its kernel is $\ker p=\cap_kM_k$, what explains exactness.

\begin{definition} The module $M$ is \emph{complete} (resp., \emph{Hausdorff}, \emph{Hausdorff-complete}) with respect to the considered filtration $M_k$ if and only if $p$ is surjective (resp., $p$ is injective, $p$ is bijective).\end{definition}

We will consider sequences of partial sums in $M$, i.e. sequences of the type $\sum_{k=0}^n m_k$, $m_k\in M$, $n\in\N$.

\begin{definition} A sequence $\sum_{k=0}^n m_k$ of partial sums in $M$ is a \emph{Cauchy sequence} if $m_k\in M_k$, for all $k$.\end{definition}

Indeed, in the topology with basis $m+M_k$, $m\in M$, $k\in \N$, the sequence $\sum_{k=r}^sm_k\in M_r$ converges to 0 if $r,s\to\infty$.

\begin{prop} If $M$ is Hausdorff (resp., complete, Hausdorff-complete), any sequence (resp., any Cauchy sequence) has at most one (resp., at least one, a unique) limit.\end{prop}

\begin{proof} If $M$ is Hausdorff, any sequence $(m(n))_n$ has at most one limit. Assume that $(m(n))_n$ converges to $m'$ and $m''$. Then, for all $k$, $m'-m''=(m'-m(n))-(m''-m(n))\in M_k$, so that $m'=m''$.

Further, if $M$ is complete, any Cauchy sequence $\sum_{k=0}^nm_k$, $m_k\in M_k$, $n\in\N$, has a limit. Indeed, the sequence $\sum_{k=0}^nm_k+M_{n+1}$, $n\in\N,$ is an element of the limit $\varprojlim_kM/M_k$ (the map $f_{n+1,n+2}$ sends $\sum_{k=0}^{n+1}m_k+M_{n+2}$ to $$\sum_{k=0}^{n+1}m_k+M_{n+1}\simeq \sum_{k=0}^nm_k+M_{n+1}\;.)$$ Hence, it is an image by $p$, i.e. there is $m\in M$ such that $\sum_{k=0}^nm_k+M_{n+1}=m+M_{n+1}$, for all $n$. In other words, $m-\sum_{k=0}^nm_k\in M_{n+1}$, or, still, the Cauchy sequence of partial sums converges to $m$.

Eventually, if $M$ is Hausdorff-complete, any Cauchy sequence of partial sums has a unique limit.\end{proof}

\subsection{Category of locally $\Z_2^n$-ringed spaces}\label{App2}

\begin{definition} A \emph{locally $\Z_2^n$-ringed space} ({\small LZRS}) is a pair $(M, \cA_M)$ made of a topological space $M$ (Hausdorff and second-countable) and a sheaf $\cA_M$ of $\Z_2^n$-commutative associative unital $\R$-algebras, such that the stalks $\cA_{m}$, $m\in M$, are local.\end{definition}

Morphisms of {\small LZRS} are maps that respect all data, i.e. they are made of a continuous base map and a sheaf morphism that respects the maximal ideals:

\begin{definition}
A \emph{morphism} of {\small LZRS} is a map $\Phi$ between two {\small LZRS} $(M, \cA_M)$ and $(N, {\cal B}_N)$, made of
\begin{enumerate}
\item a continuous map $\zvf: M\to N$,
\item a family, indexed by the open subsets $V\subset N$, of graded unital $\R$-algebra morphisms
$$ \phi_{V}^{*}: {\cal B}_N (V) \to \cA_M \lp \zvf^{-1}(V) \rp\;, $$
called \emph{pullback maps},

\begin{itemize}\item which commute with the restriction maps of the sheaves $\cA_M$ and ${\cal B}_N$,
\item and are such that the naturally induced graded unital $\R$-algebra morphisms
$$\zvf^*_m:{\cal B} _{\zvf(m)}\to \cA_{m}\;,$$ $m\in M$, respect the maximal
ideals, i.e. satisfy 
$$
 \phi^{*}_{m}( \frak{m}_{\zvf(m)}) \subset
\frak {m}_{m}\;.
$$
\end{itemize}
\end{enumerate}
\end{definition}

Locally $\Z_2^n$-ringed spaces form a category ${\tt LZRS}$ with obvious identity morphisms and composition.

\subsection{Extension and gluing of sheaves}\label{App3}

\subsubsection{Extension of a sheaf on a basis}

Let us recall that a sheaf over a topological space is completely defined by its definition on a basis of the topology: if $B$ is a basis of a topological space $M$, there is a 1:1 correspondence $\op{Sh}(M)\stackrel{\sim}{\longrightarrow}\op{Sh}(B)$ between the category of sheaves on $M$ and the category of sheaves on $B$.

Sheaves on $B$, or \emph{$B$-sheaves}, are defined exactly as sheaves on $M$, except that only open subsets in $B$ are considered. For instance, for the assignment ${\cal F}:U\mapsto {\cal F}(U)$, $U\in B$, the gluing condition reads: For any $U\in B$, any cover $(U_i)_i$, $U_i\in B$, of $U$, and any family $(f_i)_i, f_i\in {\cal F}(U_i)$, such that $f_i|_V=f_j|_V$, for any $V\in B$, $V\subset U_i\cap U_j$, there is a unique $f\in{\cal F}(U)$, such that $f|_{U_i}=f_i.$ Similarly, a morphism of $B$-sheaves is defined exactly as a morphism of sheaves.

The functor $\op{Sh}(M)\to\op{Sh}(B)$ is just the forgetful functor. Conversely, any $B$-sheaf $\cal F$ and any $B$-sheaf morphism $\zf:{\cal F}\to {\cal G}$ uniquely extend to a sheaf $\overline{\cal F}$ and a sheaf morphism $\overline{\zf}:\overline{\cal F}\to \overline{\cal G}$, respectively. For instance, the extension $\overline{\cal F}$ is (up to isomorphism) given, for any open $W\subset M$, by $$\overline{\cal F}(W)=\{(f_a)_a:f_a\in{\cal F}(U_a), U_a\in B, U_a\subset W\;\text{and}\; f_a|_V=f_b|_V, V\in B, V\subset U_a\cap U_b\}\;.$$

\subsubsection{Gluing of sheaves}

Sheaves can be glued.

More precisely, if $(U_i)_i$ is an open cover of a topological space $M$, if ${\cal F}_i$ is a sheaf on $U_i$, and if $\zf_{ji}:{\cal F}_i|_{U_i\cap U_j}\to {\cal F}_j|_{U_i\cap U_j}$ is a sheaf isomorphism such that the usual cocycle condition $\zf_{kj}\,\zf_{ji}=\zf_{ki}$ holds, then there is a unique sheaf ${\cal F}$ on $M$, such that ${\cal F}|_{U_i}\simeq{\cal F}_i$.

Actually the sheaf isomorphisms $\psi_i:{\cal F}|_{U_i}\to {\cal F}_i$ satisfy $\zf_{ji}\;\psi_i|_{Ui\cap U_j}=\psi_j|_{U_i\cap U_j}$. Uniqueness means that if there is another sheaf ${\cal F}'$ on $M$ with sheaf isomorphisms $\psi'_i:{\cal F}'|_{U_i}\to {\cal F}_i$ that satisfy the same property, then there exists a unique sheaf isomorphism $\zf:{\cal F}\to {\cal F}'$ such that $\psi'_i\;\zf|_{U_i}=\psi_i.$

The glued sheaf $\cal F$ is defined, for $U\subset M$, by $${\cal F}(U)=\{(f_i)_i\in\prod_i{\cal F}_i(U\cap U_i):\zf_{ji}(f_i|_{U\cap U_i\cap U_j})=f_j|_{U\cap U_i\cap U_j}\}\;.$$

\subsection{Partition of unity}\label{PartUnitQuotShea}

The present section should be read after Section \ref{ContMorph}.

\begin{prop} For any $\Z_2^n$-supermanifold $\cM=(M,\cA_M)$ and any open cover $(U_i)_i$ of $M$, there is a partition of unity of $\cM$ subordinated to $(U_i)_i$.\end{prop}

More precisely, there exists a locally finite cover $(V_j)_j$ of $M$, which is subordinated to $(U_i)_i$, and a family $(s_j)_j\in \cA_M^{0}(M)$, such that the projections $\zg_j:=\ze(s_j)\in\Ci_M(M)$ are nonnegative, the support $\op{supp}(s_j)$ is compact in $V_j$, and $\sum_js_j=1$. The support of a $\Z_2^n$-function $s\in\cA_M(U)$, $U\subset M$, is defined as usual as the complementary of the open subset of identical zeros of $s$ in $U$.\medskip

\begin{proof} Let $(V_j,\zg_j)_j$ be a partition of unity of the classical smooth manifold $M$ that is subordinated to $(U_i)_i$: $\op{supp}(\zg_j)$ is compact in $V_j$ and $\sum_j\zg_j=1$. We may of course assume that the $V_j$ are $\Z_2^n$-superchart domains: on $V_j$, the $\Z_2^n$-supermanifold $(M,\cA_M)$ is isomorphic to a $\Z_2 ^n$-superdomain. If $\Phi=(\phi,\phi^*)$ is this isomorphism, set $f_j=\phi^*(\zg_j\circ\zvf^{-1})\in\cA^{0}_{M}(V_j)$. It is clear that $\ze(f_j)=\zg_j$ and easily seen that $\op{supp}(f_j)=\op{supp}(\zg_j)\subset V_j$. Extend now $f_j$ by 0, so that $f_j\in\cA^{0}_M(M)$, and set $f=\sum_jf_j.$ This sum is well-defined in $\cA_M^{0}(M)$, due to local finiteness of $(V_j)_j$, and $\ze(f)=\sum_j\ze(f_j)=\sum_j\zg_j=1$. The latter implies that $f|_{V_k}$ is invertible for all $k$. When gluing these local inverses, we get a global inverse $f^{-1}\in\cA^{0}_M(M)$. It now suffices to set $s_j=f^{-1}f_j\in\cA_M^{0}(M)$.   \end{proof}

\section{Acknowledgements}

T. Covolo thanks the Luxembourgian NRF for support via AFR grant 2010-1, 786207. The research of J. Grabowski has been founded by the  Polish National Science Centre grant under the contract number DEC-2012/06/A/ST1/00256. The research of N. Poncin has been supported by the Grant GeoAlgPhys 2011-2014 awarded by the University of Luxembourg. Moreover, the authors are grateful to Sophie Morier-Genoud and Valentin Ovsienko, whose work on $\Z_2^n$-gradings \cite{MGO1}, \cite{MGO2} was the starting point of this paper. They thank Dimitry Leites for his suggestions and valuable support. N. Poncin appreciated Pierre Schapira's explanations on sheaves.

\vspace{5mm}
\small{
\noindent Tiffany COVOLO\\University of
Luxembourg\\ Campus Kirchberg, Mathematics Research Unit\\ 6, rue Richard Coudenhove-Kalergi, L-1359 Luxembourg
City, Grand-Duchy of Luxembourg
\\Email: tiffany.covolo@uni.lu \medskip

\noindent Janusz GRABOWSKI\\ Polish Academy of Sciences\\ Institute of
Mathematics\\ \'Sniadeckich 8, P.O. Box 21, 00-656 Warsaw,
Poland\\Email: jagrab@impan.pl \medskip

\noindent Norbert PONCIN\\University of Luxembourg\\
Campus Kirchberg, Mathematics Research Unit\\ 6, rue Richard Coudenhove-Kalergi, L-1359 Luxembourg
City, Grand-Duchy of Luxembourg\\Email: norbert.poncin@uni.lu}

\end{document}